\documentclass{amsart} 
\usepackage{graphicx}  
\usepackage{hyperref}
\usepackage{amsmath, amsthm, amssymb, amscd}
\usepackage{mathrsfs}
\usepackage{mathtools}
\usepackage{enumerate}
\usepackage{hyperref}
\usepackage{verbatim}
\usepackage{fixltx2e}
\usepackage{subcaption}
\usepackage{cite}
\usepackage{color}
\usepackage{esint}
\usepackage{tikz-cd}
\usetikzlibrary{shapes.geometric}
\usepackage[shortlabels]{enumitem}
 
\usepackage{bm}
\usepackage{bbm, dsfont}
\definecolor{mycolor}{rgb}{0.0, 0.75, 1.0}

\makeatletter
\newcommand{\labitem}[2]{%
\def\@itemlabel{\textbf{#1}}
\item
\def\@currentlabel{#1}\label{#2}}
\makeatother

\title[Cartesian products of Sierpi\'nski carpets]{Cartesian products of Sierpi\'nski carpets do not attain their conformal dimension}
\author{Riku Anttila}
\author{Sylvester Eriksson-Bique}
\email{sylvester.d.eriksson-bique@jyu.fi}
\address{Department of Math. and Stat.
P.O. Box 35 \\
FI-40014 University of Jyväskylä}

\thanks{RA is supported by Finnish Ministry of Education and Culture's Pilot for Doctoral Programmes (Pilot project Mathematics of Sensing, Imaging and Modelling).
SEB is supported by the Research Council of Finland via the project \emph{GeoQuantAM: Geometric and Quantitative Analysis on Metric spaces}, grant no. 354241.
}
\date{\today}

\usepackage[shortlabels]{enumitem}
 
\newtheorem{theorem}[equation]{Theorem}
\newtheorem{lemma}[equation]{Lemma}

\newtheorem{proposition}[equation]{Proposition}
\newtheorem{corollary}[equation]{Corollary}

\numberwithin{equation}{section}

\theoremstyle{definition}
\newtheorem{definition}[equation]{Definition}
\newtheorem{assumption}[equation]{Assumption}

\theoremstyle{remark}
\newtheorem{remark}[equation]{Remark}

    \newcommand*{\N}{\mathbb{N}}
    
    \newcommand*{\R}{\mathbb{R}}
    
        \DeclarePairedDelimiter\Span{\langle}{\rangle}
        
        \DeclareMathOperator{\diam}{diam}
        \DeclareMathOperator{\dist}{dist}
        \DeclareMathOperator{\Hr}{H}
        \DeclareMathOperator{\len}{len}
        \DeclareMathOperator{\Lip}{Lip}
        
        \DeclareMathOperator{\PI}{PI}
        \DeclareMathOperator{\VD}{\rm VD}
        \DeclareMathOperator{\cCap}{\rm Cap}
        \DeclareMathOperator{\CS}{\rm CS}
        \DeclareMathOperator{\cE}{\mathcal{E}}
    
        \DeclareMathOperator{\cF}{\mathcal{F}}

        \DeclareMathOperator{\Mod}{Mod}

        \DeclarePairedDelimiter\abs{\lvert}{\rvert}
        \DeclarePairedDelimiter\norm{\lVert}{\rVert}
        
        \DeclarePairedDelimiter{\ceil}{\lceil}{\rceil}
        
\def\vint_#1{\mathchoice%
          {\mathop{\kern 0.2em\vrule width 0.6em height 0.69678ex depth -0.58065ex
                  \kern -0.8em \intop}\nolimits_{\kern -0.4em#1}}%
          {\mathop{\kern 0.1em\vrule width 0.5em height 0.69678ex depth -0.60387ex
                  \kern -0.6em \intop}\nolimits_{#1}}%
          {\mathop{\kern 0.1em\vrule width 0.5em height 0.69678ex
              depth -0.60387ex
                  \kern -0.6em \intop}\nolimits_{#1}}%
          {\mathop{\kern 0.1em\vrule width 0.5em height 0.69678ex depth -0.60387ex
                  \kern -0.6em \intop}\nolimits_{#1}}}
\def\vintslides_#1{\mathchoice%
          {\mathop{\kern 0.1em\vrule width 0.5em height 0.697ex depth -0.581ex
                  \kern -0.6em \intop}\nolimits_{\kern -0.4em#1}}%
          {\mathop{\kern 0.1em\vrule width 0.3em height 0.697ex depth -0.604ex
                  \kern -0.4em \intop}\nolimits_{#1}}%
          {\mathop{\kern 0.1em\vrule width 0.3em height 0.697ex depth -0.604ex
                  \kern -0.4em \intop}\nolimits_{#1}}%
          {\mathop{\kern 0.1em\vrule width 0.3em height 0.697ex depth -0.604ex
                  \kern -0.4em \intop}\nolimits_{#1}}}

\newcommand{\kint}{\vint}

\newcommand{\aveint}[2]{\mathchoice%
          {\mathop{\kern 0.2em\vrule width 0.6em height 0.69678ex depth -0.58065ex
                  \kern -0.8em \intop}\nolimits_{\kern -0.45em#1}^{#2}}%
          {\mathop{\kern 0.1em\vrule width 0.5em height 0.69678ex depth -0.60387ex
                  \kern -0.6em \intop}\nolimits_{#1}^{#2}}%
          {\mathop{\kern 0.1em\vrule width 0.5em height 0.69678ex depth -0.60387ex
                  \kern -0.6em \intop}\nolimits_{#1}^{#2}}%
          {\mathop{\kern 0.1em\vrule width 0.5em height 0.69678ex depth -0.60387ex
                  \kern -0.6em \intop}\nolimits_{#1}^{#2}}}

\begin{document}

\maketitle

\begin{abstract}
It is a long-standing open question to determine whether the Sierpi\'nski carpet attains its conformal dimension or not.
While this problem remains unresolved, we prove that Cartesian products $\mathbb{S}^k$, where $\mathbb{S}$ is the Sierpi\'nski carpet and $k \geq 2$, do not attain their conformal dimension.
Our approach is based on the Sobolev spaces and energy measures on $\mathbb{S}$ -- constructed by Shimizu, Kigami, and Murugan and Shimizu -- together with a certain singularity result of energy measures from the theory of analysis on fractals.
This work formulates a general non-attainment result of conformal dimension for product metric spaces $X^k$ for $k \geq 2$ in terms of self-similarity and energy measures of the factor $X$. It applies, in particular, to the cases where $X$ is the Sierpi\'nski carpet, the Sierpi\'nski gasket, the Menger sponge, and the Laakso diamond.

\end{abstract}

\section{Introduction}

\subsection{Background}

The \emph{Ahlfors regular conformal dimension} of a metric space $X$ is the infimum of Hausdorff dimensions of all metric spaces $Y$ that are Ahlfors regular and quasisymmetrically homeomorphic to $X$. This notion originates in the work of Bourdon and Pajot \cite{BourdonPajot} and is a variant of the \emph{conformal dimension} introduced by Pansu \cite{Pansu}.
Precise definitions are given in Section \ref{sec:preli}, and a more comprehensive treatment can be found in the book by Mackay and Tyson \cite{MT}.
For brevity, we refer to the Ahlfors regular conformal dimension simply as the conformal dimension throughout.

A central concept in the literature on conformal dimension is the \emph{attainment problem}, which asks whether the infimum of Hausdorff dimensions of $Y$, as described in the previous paragraph, is a minimum.
In fact, this problem turns out to be related to a remarkably large number of questions in different fields.
To name a few notable ones, Bonk and Kleiner showed that the well-known Cannon's conjecture can be reformulated in terms of an attainment problem of conformal dimension for the visual boundary of hyperbolic groups \cite{BK05}. 

There are further surprising connections between attainment problems and existence problems for exotic metrics and measures. In the work of the second author and G. C. David \cite{GuySylvester}, it was observed that the attainment for a certain Pillow space implies the existence of an analytically 1-dimensional plane in the sense of Cheeger \cite{Cheeger}.
Whether such a plane exists was posed by Kleiner and Schioppa \cite{KleinerSchioppa} and remains an open problem.

Lastly, Bonk and Kleiner \cite{BK05} observed that attainment implies, for a large class of group boundaries, that they are quasisymmetrically equivalent to a Loewner space in the sense of Heinonen and Koskela \cite{HK}. See also the discussion in Kleiner's  ICM survey \cite{KleinerICM}. A similar result applies to a large class of self-similar, combinatorially Loewner spaces by \cite[Section 1.6]{CEB}.
For motivation and references about Loewner spaces, or more generally PI-spaces, see \cite{HKST,bjorn2011nonlinear,He}.

Despite the extensive research, it is an unfortunate fact that the attainment problem has been resolved only for a very limited class of examples.
One source of this difficulty appears to be that the geometries of the metric spaces for which the attainment problem is of interest are highly intricate. As a result, the existing tools do not seem capable of determining attainment when only non-explicit information is available. A way forward was suggested by Murugan and Shimizu who studied the relationship of energy measures to attainment problems \cite{MuruganShimizu}. Here, we take this concept a step further by showing how such tools can be used to determine attainment. 

An interested reader can find from \cite{MT} a list of examples where the attainment have been deduced.
See also the recent work by the authors and Rainio \cite{anttila2024constructions,RainioLaakso26} which provides a quite rich class of geometries where the attainment can be realized in a non-trival way, as-well as examples where the attainment fails. While these examples are varied, we note that the topologies of the examples in \cite{anttila2024constructions,RainioLaakso26} are somewhat repetitive; most of them are homeomorphic to the universal Menger curve by \cite{Anderson-1,Anderson-2}. It would be desirable to have more varied examples, such as manifolds and topological carpets, where the attainment problem can be decided. 

\subsection{Main result}
A notorious example for which the attainment of the conformal dimension remains open is the \emph{Sierpi\'nski carpet} \cite[Problem 6.2]{BK05}.
The value of the conformal dimension is also unknown \cite[Open problem 15.22]{He} but quite impressive estimates are available due to Kwapisz \cite{Kwapisz}.
The hypothetical realization of the attainment was recently studied by Murugan and Shimizu \cite[Section 9]{MuruganShimizu}.

While the present work does not resolve the attainment problem for the Sierpi\'nski carpet, our main result is a proof of the following non-attainment result.

\begin{theorem}\label{main:Carpet}
    If $\mathbb{S}$ is the standard Sierpi\'nski carpet and $k \geq 2$, then the Cartesian product $\mathbb{S}^k$ does not attain its conformal dimension.
\end{theorem}

Kleiner had conjectured in \cite[Conjecture 7.5]{KleinerICM} that approximately self-similar and \emph{combinatorially Lowner} metric spaces attain their conformal dimension (see \cite{Clais} for a definition that is revised from \cite{BourK}). By \cite[Section 4.3]{BourK} or \cite[Remark 4.5 and Theorem 8.3]{ReplacementGraph24}, the Cartesian product $\mathbb{S}^k$ satisfies the combinatorial Loewner property. Theorem \ref{main:Carpet}, therefore, provides new counterexamples to Kleiner's conjecture \cite[Conjecture 7.5]{KleinerICM}. While the conjecture was first resolved in \cite{anttila2024constructions}, all counterexamples therein are homeomorphic to the Menger curve. It was posed in \cite{GuySylvester} if one can construct counterexamples with other topologies. The present work provides the first counterexamples of arbitrary topological dimension. 

Our methods are not restricted to the Sierpi\'nski carpet, and apply to many self-similar spaces. See the end of the introduction for a detailed discussion.

\subsection{Main techniques}
The starting point of the present work is the following elementary observation that illustrates a rigidity of product spaces in quasiconformal geometry.
Let us note that the we are not the first one to exploit similar ideas. See the work of Väisälä from 1989 \cite[Section 5]{VaisalaCylinder89}, and the more recent study on Carnot groups by Le Donne and Xie \cite{LedonneXie16}.
See also \cite[Example 4.1.9]{MT}

Suppose that $f : X \to Y$ is a quasisymmetry. Then it is false in general that the induced homeomorphism $\tilde{f} : X^2 \to Y^2$ where $\tilde{f}(x_1,x_2) := (f(x_1),f(x_2))$ is a quasisymmetry.
The mapping $f : \mathbb{C} \to \mathbb{
C}$ where $f(z) := \abs{z}z$ is one such example.
In fact, ignoring some obvious examples like bi-Lipschitz or snowflake maps, $\tilde{f}$ is virtually never a quasisymmetry in practical situations.
In what follows, the attainment problem for $X$ and $X^2$ are separate questions.
We strongly encourage the reader to carry out some simple computations and to draw some pictures in order to be convinced about this fact.

Based on the discussion above, we formulate the following heuristic principle which, aside from examples like $\R$, seems to be tentative:
\emph{Regardless of whether $X$ attains its conformal dimension, the hypothetical attaining structure of $X^k$ should not have a product structure arising from $X$.}

In order to exploit this vague intuition, we develop an analytic machinery based on \emph{energy measures} and \emph{blow-up analysis}. While these tools may appear somewhat heavy and unnatural for our problem, we are not aware of any conceptually different approach that would prove Theorem \ref{main:Carpet}.
Our ideas are inspired by the two pioneering works, Kajino and Murugan on the conformal walk dimension \cite{KajinoMurugan}, and also Murugan and Shimizu reformulating the attainment of the conformal dimension on $\mathbb{S}$ in terms of energy measures \cite{MuruganShimizu}.
In the remainder of the introduction, we sketch the proof of Theorem \ref{main:Carpet} and take $k = 2$ for simplicity.
We argue by contraposition, i.e., assume that the attainment is realized at $(\mathbb{S}^2,\rho)$ and denote its conformal dimension by $Q$.

We first describe the tools in our blow-up analysis. Note that $\mathbb{S}^2$ as a subset of $\R^{4}$ is an attractor of an IFS $\{\Phi_D\}_{D \in \mathcal{D}}$ where each $\Phi_D$ is a composition of a translation and a contraction with contraction ratio equal to $1/3$.
We consider limits of the metric spaces $(\mathbb{S}^2,\rho_{n})$ where $\rho_n$ is a normalized pull-back metric of a composition $\Phi_{D_1} \circ \cdots \circ \Phi_{D_n}$. Note that $\Phi_D$ is not necessarily similitude in the metric $\rho$, but we can derive suitable uniform estimates that are required to use Gromov's compactness theorem. As a result, we obtain a limit $(\mathbb{S}^2,\tilde{\rho})$ in the Gromov--Hausdorff sense.
As it is shown in Lemma \ref{lemma:Blow-up-2}, $(\mathbb{S}^2,\tilde{\rho})$ is also an attaining metric space.

Next, we discuss the $Q$-energy measures on $\mathbb{S}^2$, and we begin by considering the analogous objects on the factor $\mathbb{S}$ that are provided by the work of Murugan and Shimizu \cite{MuruganShimizu}.
In this specific case, we could also refer to the earlier works of Kigami \cite{kigami} and Shimizu \cite{shimizu} due to the same explanation as in Remark \ref{rem:Menger}.
Let $Q$ be the conformal dimension of $\mathbb{S}^2$ and $\mathcal{E}_\mathbb{S} : L^Q(\mathbb{S},\mu_\mathbb{S}) \to [0,\infty]$ be the $Q$-energy form in \cite[Theorem 1.1]{MuruganShimizu} where $\mu_\mathbb{S}$ is the $\log(8)/\log(3)$-Hausdorff measure of $\R^2$. We also let $\mathcal{F}_\mathbb{S} := \mathcal{E}_\mathbb{S}^{-1}([0,\infty))$ be the Sobolev space, and finally $\{\Gamma_\mathbb{S}\Span{f}\}_{f \in \mathcal{F}_\mathbb{S}}$ be the $Q$-energy measures in \cite[Theorem 1.2]{MuruganShimizu}.

Roughly speaking, $\mathcal{E}_\mathbb{S}$ is an abstract counterpart of the Dirichlet energy $f \mapsto \norm{ \nabla f }_{L^Q(\R^n)}^Q$, which in particular means that $\mathcal{F}_\mathbb{S}$ plays the role of the first-order Sobolev space $W^{1,Q}(\R^n)$. The purpose of the energy measures $\Gamma_\mathbb{S}\Span{f}$ is to be a counterpart of the ``localized energies'', i.e. the Radon measures $A \mapsto \norm{ \nabla f }_{L^Q(A)}^Q$.
The existence of these energy measures is a non-trivial fact, and they are, unfortunately, rather inexplicit. To the best of our knowledge, the currently available tools do not yield a single non-constant function $f \in \mathcal{F}_\mathbb{S}$. All constructions known to us rely heavily on abstract compactness arguments.

The construction of energy measures on the product $\mathbb{S}^2$ is based on the energy measures on the factor $\mathbb{S}$ and a method introduced by Strichartz \cite{Strichartz}.
Let $\mu = \mu_\mathbb{S} \otimes \mu_\mathbb{S}$ be the product measure on $\mathbb{S}^2$, which is easily seen to be comparable to the $\log(64)/\log(3)$-Hausdorff measure.
Consider the $Q$-energy form $\mathcal{E} : L^Q(\mathbb{S}^2,\mu) \to [0,\infty]$ given by the equation
\begin{equation*}
    \mathcal{E}(u) := \int_{\mathbb{S}} \mathcal{E}_\mathbb{S}( u(x,\,\cdot\, ) ) \, \text{d}\mu_{\mathbb{S}}(x) +  \int_{\mathbb{S}} \mathcal{E}_\mathbb{S}( u( \,\cdot\,,y ) )\, \text{d}\mu_{\mathbb{S}}(y). 
\end{equation*}
Here $u(x,\,\cdot\, )(y) := u(x,y)$ and $u( \,\cdot\,,y )(x) := u(x,y)$ are the sections of $u$.
The Sobolev space is denoted by $\mathcal{F} := \mathcal{E}^{-1}([0,\infty))$, and each $u \in \mathcal{F}$ is assigned the $Q$-energy measure
\begin{equation*}
    \Gamma\Span{u}(\Omega) := \int_{\mathbb{S}}\Gamma_\mathbb{S}\Span{u(x,\, \cdot\,)}(\Omega_x)\, \text{d}\mu_\mathbb{S}(x) + \int_{\mathbb{S}}\Gamma_\mathbb{S}\Span{u(\, \cdot\,,y)}(\Omega^y)\, \text{d}\mu_\mathbb{S}(y)
\end{equation*}
where $\Omega_x := \{ y : (x,y) \in \Omega \}$ and $\Omega^y := \{ x : (x,y) \in \Omega \}$ are the slices of $\Omega$.

\subsection{Sketch of the proof}
The proof scheme of Theorem \ref{main:Carpet} is as follows.
The first step is to show that the attaining metric space $(\mathbb{S}^2,\rho)$ must have the following product structure.
If a Radon measure $\Lambda$ on $\mathbb{S}$ is a \emph{minimal energy dominant measure} of $\mathcal{E}_{\mathbb{S}}$ in the sense of Hino \cite{Hino10} (see Section \ref{sec:Dirichlet} for the definition) and $\nu$ is the $Q$-Hausdorff measure of $(\mathbb{S}^2,\rho)$, then $(\mathbb{S}^2,\rho)$ has a product structure:
\begin{equation}\label{Intro:eq-2}
    \nu \text{ is mutually absolutely continuous with } \Lambda \otimes \mu_{\mathbb{S}} + \mu_{\mathbb{S}} \otimes \Lambda.
\end{equation}
This follows from Corollary \ref{cor:Lambda-minimal} and a ``conformal invariance'' of the $Q$-energy form, stated in Proposition \ref{prop:Conf-Inv} and Corollary \ref{cor:nu-minimal}.
Our strategy is to find a contradiction in \eqref{Intro:eq-2} using the heuristic principle mentioned above.

Then consider $h := \text{d}(\Lambda \otimes \mu_{\mathbb{S}})/\text{d}\nu$ and its density point $x$ where $h(x) > 0$. We proceed by blowing-up at such point, and obtain a metric space $(\mathbb{S}^2,\tilde{\rho})$, quasisymmetric to $(\mathbb{S}^2,\rho)$.
We also show that, if $\nu_n$ are the $Q$-Hausdorff measures in the blow-up sequence $(\mathbb{S}^2,\rho_n)$, then a subsequence of $\nu_n$ converges weakly to measure $\tilde{\nu}$ satisfying the properties
\begin{equation}\label{Intro:eq-3}
    \tilde{\nu}(B_{\tilde{\rho}}(x,r)) \sim r^Q \text{ and } \tilde{\nu} \ll \tilde{\Lambda} \otimes \mu_{\mathbb{S}}.
\end{equation}
for some Radon measure $\tilde{\Lambda}$.
The first condition in \eqref{Intro:eq-3} is obtained by showing that $\nu_n(B_{\rho_n}(x,r))$ is comparable to $r^Q$ with constant independent of $n$. The second condition in \eqref{Intro:eq-3} is somewhat delicate because weak limits are generally unsuitable for studying absolute continuity.
Nevertheless, here we can exploit a fact that
\[
    \textrm{m}_{1,n} \rightharpoonup \tilde{\textrm{m}}_1 \text{ and } \textrm{m}_{2,n} \rightharpoonup \tilde{\textrm{m}}_2 \text{ implies } \textrm{m}_{1,n} \otimes \textrm{m}_{2,n} \rightharpoonup \textrm{m}_{1} \otimes \textrm{m}_{2},
\]
which always holds under mild topological assumptions.
Hence, the second conditions in \eqref{Intro:eq-3} follows because we blow-up at a density point of $h$.

Now, \eqref{Intro:eq-3} implies that $(\mathbb{S}^2,\tilde{\rho})$ is also an attaining metric space, and that its $Q$-Hausdorff measure is comparable to $\tilde{\nu}$. But since \eqref{Intro:eq-2} is valid for \emph{any} attaining metric space, we must have that
\[
    \mu_{\mathbb{S}} \otimes \Lambda \ll \Lambda \otimes \mu_{\mathbb{S}} + \mu_{\mathbb{S}} \otimes \Lambda \ll \tilde{\nu} \ll \tilde{\Lambda} \otimes \mu_{\mathbb{S}}.
\]

Up to this point, each step in the proof would work if, for instance, $\mathbb{S}$ is replaced by the unit interval $[0,1]$, for which the conclusion of the theorem is obviously false.
But now, we use a certain singularity property carried by energy measures on many fractal spaces, including the Sierpi\'nski carpet.
More specifically, if $\Lambda$ is a minimal energy dominant measure, then $\Lambda \perp \mu_\mathbb{S}$.
See Section \ref{sec:Singular} for more detailed discussion about the singularity.
This together with absolute continuity in the previous display implies $\Lambda = 0$, which clearly contradicts \eqref{Intro:eq-2}. The attaining structure therefore cannot exist, and this completes the proof of Theorem \ref{main:Carpet}.

\subsection{Further examples}
As discussed above, the proof of Theorem \ref{main:Carpet} relies on two main techniques: blow-ups and energy measures.
The existence of blow-ups are ensured by rather mild self-similarity; essentially all self-similar sets in Euclidean spaces admit them.
The existence of suitable energy measures is far more non-trivial. Nevertheless, there is by now a substantial body of literature in analysis on fractals that provides such examples. See, for instance, \cite{shimizu,kigami,p-Gasket,KigamiOta,AEBS25}, as well as the recent survey \cite{energyform-survey}.
Below, we provide a list of examples of metric spaces $X$ for which Theorem \ref{main:Carpet} holds if we replace $\mathbb{S}$ with $X$.
A precise generalized statement is provided in Theorem \ref{thm:non-attainment}, and see Section \ref{subsec:Remarks} for more detailed discussion.

\begin{enumerate}
    \item $X$ is the Menger sponge, or more generally any generalized Sierpi\'nski carpet in the sense of \cite{shimizu}, where also the relevant energy form is constructed.
    \item $X$ is the Laakso diamond space, where the energy is constructed in \cite{AEBS25}. More generally, if $X$ is any Laakso type space in the sense of \cite{AEBS25}, assuming that for every $p \in (1,\infty)$ its $p$-walk dimension is strictly larger than $p$.  
    \item $X$ is the Sierpi\'nski gasket, with the energy constructed in \cite{kigami}.
\end{enumerate}

Finally, we note that the construction of the energy measures on $\mathbb{S}$ discussed above generalizes to every $p \in (1,\infty)$.
For our application to the non-attainment, however, it is essential that we consider the $Q$-energy where $Q$ is the conformal dimension of $\mathbb{S}^2$.
For instance, the ``conformal invariance'' of the $Q$-energy, stated in Proposition \ref{prop:Conf-Inv} and Corollary \ref{cor:nu-minimal}, should be false when $Q$ is some other value.
This, nevertheless, does not mean that $p$-energies for other values of $p$ are not interesting.
In fact, our method for the non-attainment demands a study of the $p$-energy on the factor $\mathbb{S}$, which is not conformally invariant in the sense above; $Q$ is the conformal dimension of the product $\mathbb{S}^2$ but not of the factor $\mathbb{S}$.
The results of the work, therefore, indicates that the aforementioned studies of energy measures, for general $p \in (1,\infty)$, should be applicable in the future studies on conformal dimension.

\subsection*{Organization of the paper}
Section \ref{sec:preli} covers the necessary material on metric spaces and conformal dimension. In Section \ref{sec:Dirichlet}, we introduce the general framework of $p$-Dirichlet spaces that were developed in \cite{EID2025,ResistanceConjecture}.

In Section \ref{Sec:Confdim}, we establish a connection between the attainment problem with properties of $p$-Dirichlet spaces.
Then in Section \ref{sec:Singular}, we provide the details about the singularity of energy measures.

Finally, in Section \ref{sec:NonAttain}, we prove the main theorem about the non-attainment, Theorem \ref{main:Carpet}, and its generalization Theorem \ref{thm:non-attainment}. A detailed discussion about the hypotheses is provided in the very end of the paper.

\section{Preliminary}\label{sec:preli}

\subsection{Metric spaces}
We recall some standard terminology from metric geometry, and refer to \cite{He,bbi} for further details.
Let $(X,d)$ be a complete metric space. If $A\subseteq X$ is a non-empty subset, write $\diam(A):=\sup_{x,y \in A} d(x,y)$. 
For $Q > 0$, the  \emph{$Q$-dimensional Hausdorff measure} of a subset $A \subseteq X$ is
\[
    \mathcal{H}^Q(A):=\lim_{\delta\to 0}\mathcal{H}_\delta^Q(A).
\]
Here $\mathcal{H}^Q_\delta$ is the \emph{$Q$-dimensional Hausdorff $\delta$-content}, given by
\[
    \mathcal{H}^Q_\delta(A):=\inf\bigg\{\sum_{i=1}^\infty \diam(A_i)^Q\; : \;A\subseteq\bigcup_{i=1}^\infty A_i,\;\diam(A_i)\leq\delta \bigg\}.
\]
The \emph{Hausdorff dimension} of \((X,d)\) is the value
    \[
        \dim_{\mathrm{H}}(X,d):=\inf \left\{ Q>0 : \mathcal{H}^Q(X)=0\right\}.
    \] 
We say that $(X,d)$ is \emph{$Q$-Ahlfors regular} if there is $C \geq 1$ such that
\begin{equation}\label{eq:AR}
    C^{-1}r^Q \leq \mathcal{H}^Q(B(x,r)) \leq Cr^Q \text{ for all } x \in X \text{ and } r \in (0,2\diam(X)].
\end{equation}
Here $B(x,r)$ denotes the open ball
\[
    B(x,r) := \{ y \in X : d(x,y) < r \}.
\]
An inflation of a ball $B=B(x,r)$ by a constant $C>0$ is often denoted $CB:=B(x,Cr)$. The Hausdorff dimension of a $Q$-Ahlfors regular metric space is equal to $Q$. We refer to \cite[pages 61-62]{He} for details.

\subsection{Curves}
Let $(X,d)$ be a compact metric space.
A \emph{curve} in $X$ is a continuous function $\gamma : [0,1] \to X$.
Its \emph{length} is the value
\[
    \len(\gamma) := \sup \sum_{k = 1}^{N - 1} d(\gamma(t_k),\gamma(t_{k+1}))
\]
where the supremum is taken over all increasing sequences $\{ t_k \}_{k=1}^N \subseteq [0,1]$.
We say that $(X,d)$ is \emph{geodesic} if for every $x,y \in X$ there is a curve $\gamma$ such that $\gamma(0) = x$, $\gamma(1) = y$ and $\len(\gamma) = d(x,y)$.

When convenient, we regard a curve $\gamma$ in $X$ as the subset of $X$ by identifying it as the image $\gamma([0,1])$.
We say that $(X,d)$ is of \emph{bounded turning} if there is $K \geq 1$ satisfying the following. For every $x,y \in X$ there is a curve $\gamma$ such that $x,y \in \gamma$ and $\diam(\gamma) \leq K d(x,y)$.

\subsection{Conformal dimension}\label{subsec:Confdim}
We recall the necessary terminology related to the conformal dimension, and see Mackay--Tyson for further literature \cite{MT}.
Let $d$ and $\rho$ be two metrics on $X$. We say that $d$ is \emph{quasisymmetric} to $\rho$ if there is a homeomorphism $\eta\colon[0,\infty)\to[0,\infty)$ such that, for all triples $
x,y,z\in X$ with $x \neq z$,
    \[
    \frac{d(x,y)}{d(x,z)}\leq\eta\bigg(\frac{\rho(x,y)}{\rho(x,z)}\bigg).
    \]
Quasisymmetry is an equivalence relation of metrics by \cite[Chapter 10]{He}.
We sometimes say that $d$ is $\eta$-quasisymmetric to $\rho$ when the function $\eta$ is relevant.
The \emph{Ahlfors regular conformal dimension} of a metric space $(X,d)$ is the infimum of exponents $Q > 0$ such that there is a metric $\rho$ on $X$, quasisymmetric to $d$, which is $Q$-Ahlfors regular.
For simplicity, we call this value the \emph{conformal dimension} of $(X,d)$.
We say that a metric space $(X,d)$ \emph{attains} its conformal dimension if there is a metric $\rho$, quasisymmetric to $d$, such that $\rho$ is $Q$-Ahlfors regular and $Q$ is the conformal dimension of $(X,d)$.

The conformal dimension is often studied using the discrete modulus. We review the necessary tools for this approach.
Let $(X,d)$ be a compact metric space, $\mathcal{U}$ be an open covering of $X$, and $\mathcal{C}$ be a family of curves in $X$.
We say that a function $\rho : \mathcal{U} \to [0,\infty)$ is \emph{$\mathcal{C}$-admissible} if
\[
    \sum_{ \substack{U \in \mathcal{U}\\
    U \cap \gamma \neq \emptyset }} \rho(U) \geq 1
\]
for all $\gamma \in \mathcal{C}$.
For $p \in [1,\infty)$, we define the \emph{discrete $p$-modulus} as the value
\[
    \Mod_p(\mathcal{C},\mathcal{U}) := \min_{\rho} \sum_{U \in \mathcal{U}} \rho(U)^p
\]
where the minimum is taken over all $\mathcal{C}$-admissible $\rho : \mathcal{U} \to [0,\infty)$.
When $p > 1$ the minimizer $\rho$ exists and is unique by standard convex optimization theory.

\subsection{Loewner spaces}
We recall some important concepts from analysis on metric space, and see the standard references \cite{bjorn2011nonlinear,HKST,He}.
Let $(X,d)$ be a complete metric space and $\mu$ be a Radon measure on $(X,d)$. We say that $\mu$ is \emph{doubling} if there is $D \geq 1$ such that
\[
    0 < \mu(B(x,2r)) \leq D\mu(B(x,r)) < \infty \text{ for all } x \in X \text{ and } r > 0.
\]
If $f \in L^1(X,\mu)$ and $A \subseteq X$ is a Borel set with $\mu(A) \in (0,\infty)$, we define the \emph{integral average}
\[
    f_A := \kint_A f\, d\mu := \frac{1}{\mu(A)} \int f\, d\mu.
\]
For $p \in [1,\infty)$, we say that $(X,d,\mu)$ satisfies the \emph{$(1,p)$-Poincar\'e inequality} if there are $C,\sigma \geq 1$ such that for all Lipschitz functions $f : X \to \R$ it holds that
\[
    \kint_{B(x,r)} \abs{f - f_{B(x,r)}} \, d\mu \leq C r \left( \kint_{B(x,\sigma r)} \Lip(f)^p \, d\mu \right)^{\frac{1}{p}}.
\]
Here $\Lip(f)$ denotes the \emph{pointwise Lipschitz constant}
\[
    \Lip(f)(x) := \limsup_{r \to 0^+} \sup_{y \in B(x,r)} \frac{\abs{f(x) - f(y)}}{d(x,y)}.
\]
Finally, given $Q > 1$, we say that a metric space $(X,d)$ is a $Q$-\emph{Loewner space} if it is $Q$-Ahlfors regular and $(X,d,\mathcal{H}^Q)$ satisfies the $(1,Q)$-Poincar\'e inequality.

\section{Dirichlet spaces}\label{sec:Dirichlet}
The main techniques used in the proof of non-attainment require a sufficiently general setting in which a first-order calculus on metric spaces can be developed.
It should, at least, encompass the Sobolev spaces on the Sierp\'nski carpet constructed in \cite{MuruganShimizu,shimizu,kigami}, and the analogous objects on the product spaces.
The more standard approaches in analysis on metric spaces, popularized by e.g. Cheeger \cite{Cheeger}, Haj{\l}asz \cite{Hajlasz} and Shanmugalingam \cite{Shanmun}, and also Ambrisio--Gigli--Savare \cite{AmbrosioGigliSavare14}, do not appear to provide a sufficiently flexible setting for our purposes.
Instead, we work with $p$-Dirichlet spaces introduced by the second author and Murugan \cite{EID2025}.

\subsection{Definition}\label{d:dir-space}
        Let $p \in (1,\infty)$. We call $(X,d,\mu, \mathcal{E}, \mathcal{F}, \Gamma)$ a \emph{$p$-Dirichlet space} if the conditions (1)-(7) below hold.
		\begin{enumerate}  
			\item \textbf{Locally compact space}: $(X,d,\mu)$ is a complete, locally compact metric space equipped with a Radon measure $\mu$.
			\item \textbf{Completeness:} 
            The space $\mathcal{F}$  is a subspace of $L^p(X,\mu)$ and
			$\mathcal{E}: \mathcal{F} \to [0,\infty)$ is a non-negative function such that $\mathcal{F}$ is a Banach space when equipped with the norm $\norm{f}_{\mathcal{F}}=(\norm{f}_{L^p}^p+\mathcal{E}(f))^{1/p}$.
			\item \textbf{Homogeneity:} For all $f\in \mathcal{F}$ there exists a finite non-negative Borel measure $\Gamma \langle f \rangle$ on $X$ such that $\Gamma \langle f \rangle(X)= \mathcal{E}(f)$   and for all $\lambda \in \mathbb{R}$
			\[
			\Gamma\langle \lambda f\rangle = |\lambda|^p \Gamma\langle f\rangle.
			\]
			The measure $\Gamma\langle f \rangle$ associated to $f \in \mathcal{F}$ is called the \emph{energy measure of $f$}.
			\item \textbf{Triangle inequality:} For every $f,g\in \mathcal{F}$ and every Borel set $A\subseteq X$, 
			\begin{equation*} \label{e:def-sublinear}
				\Gamma\langle f+g\rangle(A)^{\frac{1}{p}}\le \Gamma\langle f\rangle(A)^{\frac{1}{p}}+\Gamma\langle g\rangle(A)^{\frac{1}{p}}
			\end{equation*}
			\item \textbf{Lipschitz contractivity:} For all $f\in \mathcal{F}$ and all  $1$-Lipschitz functions $\varphi : \R \to \R$ with $\varphi(0)=0$  we have $\varphi \circ f \in \mathcal{F}$ and $\Gamma\langle \varphi \circ f\rangle \leq \Gamma\langle f\rangle$.
			\item \textbf{Strong locality:} For every $f\in \mathcal{F}$ and an open set $A\subseteq X$, if there is $c \in \R$ such that $(f - c\mathds{1}_A) = 0$ $\mu$-a.e., then $\Gamma\langle f\rangle(A)=0$. 
			\item \textbf{Weak lower semicontinuity:} For every $f\in L^p(X,\mu)$, and for any sequence of functions $\{f_i\}_{i = 1}^\infty$ in $\mathcal{F}$ such that $f_i\to f \in L^p(X,\mu)$ with $\sup_{i\in \mathbb{N}} \mathcal{E}(f_i)<\infty$, then $f\in \mathcal{F}$ and
			\begin{equation*} \label{e:def-lsc}
				\Gamma\langle f \rangle(A)\le \liminf_{i\to \infty} \Gamma\langle f_i \rangle(A)
			\end{equation*}
			for every Borel set $A\subseteq X$. 
		\end{enumerate}

    \subsection{Poincar\'e and capacity}
    We call $\Psi : [0,\infty) \to [0,\infty)$ a \emph{scale function} if there is $C \geq 1$ such that
    \[
      C^{-1}\Psi(r) \leq \Psi(2r) \leq C \Psi(r) \text{ for all } r > 0 \text{ and } \Psi^{-1}(\{0\}) = \{0\}.
    \]
    When convenient, we sometimes write $\Psi(B) = \Psi(r)$ where $B = B(x,r)$.
    Given $p \in (1,\infty)$ we call $(X,d,\mu,\mathcal{E},\mathcal{F},\Gamma)$ a \emph{$(p,\Psi)$-Poincar\'e--Dirichlet space} if it satisfies the conditions (1)-(7) above, and also the following.
    \begin{enumerate}\setcounter{enumi}{7}
        \item \textbf{Volume doubling:} $(X,d,\mu)$ is volume doubling.
        \item \textbf{Poincar\'e inequality:} There are $C,\sigma \geq 1$ such that, for all open balls $B$ and $f\in \cF$, we have
    \begin{equation}\label{eq:PI}
     \int_{B} \abs{f - f_{B}}^p \, \text{d}\mu \leq C \Psi(B) \int_{\sigma B} \, \text{d}\Gamma\Span{f}.
    \end{equation}
        \item \textbf{Upper capacity:} There is $C \geq 1$ such that, for every open ball $B \subseteq X$, there exists $\varphi \in \cF \cap C(X)$ with $\varphi|_B=1$, $\varphi|_{X\setminus 2B}=0$ and
    \[
    \mathcal{E}(\varphi) \leq C\frac{\mu(B)}{\Psi(B)}.
    \]
    \end{enumerate}
    These conditions are abbreviated as $\VD$, $\PI(\Psi)$ and $\cCap(\Psi)$, respectively.
    When $\Psi(r)=r^\beta$ for some $\beta > 0$, we often use the terminology $(p,\beta)$-Poincar\'e--Dirichlet space and write $\PI(\beta), \cCap(\beta)$ for the conditions above.
    
    In \cite{ResistanceConjecture} the scale functions were allowed to depend also on the spatial variable $x \in X$. In the present work, however, it is essential that $\Psi$ is radial. See the proof of Theorem \ref{thm:Z}.

\subsection{Preliminary lemmas}\label{subsec:BasicProp}
The following lemma is proven in \cite[Lemma 4.2]{ResistanceConjecture}.
The assumptions on $\Psi$ are slightly different, but same the proof applies.
\begin{lemma}\label{lemma:POU}
Let $\Psi$ be a scale function and $(X,d,\mu, \mathcal{E}, \mathcal{F}, \Gamma)$ be a $p$-Dirichlet space satisfying $\textup{VD}$ and $\cCap(\Psi)$.
Let $\mathcal{B}$ be a collection of open balls so that it covers $X$ and there is $D \geq 1$ satisfying
\[
    \sum_{B \in \mathcal{B}} \mathds{1}_{2B} \leq D \mathds{1}_X.
\]
Then there is $C \geq 1$, which depends on $\mathcal{B}$ only through the constant $D$, and a partition of unity
$\{ \psi_{B} \}_{B \in \mathcal{B}} \subseteq \mathcal{F} \cap C(X)$, satisfying the following.
    \begin{enumerate}
        \vspace{2pt}
        \item $\sum_{B \in \mathcal{B}} \psi_{B}(x) = 1$ and $0 \leq \psi_{B}(x)\leq 1$ for all $B \in \mathcal{B}$ and $x \in X$.
        \vspace{2pt}
        \item It holds for all $B \in \mathcal{B}$ that $\psi_{B}|_{X \setminus 2B} = 0$.
        \vspace{2pt}
        \item It holds for every $B \in \mathcal{B}$ that
        \[
        \mathcal{E}(\psi_B) \leq C\frac{\mu(B)}{\Psi(B)}.
        \]
    \end{enumerate}
\end{lemma}

Let $(X,d,\mu,\cE,\cF,\Gamma)$ be a $p$-Dirichlet space. We say that a Radon measure $\Lambda$ on $(X,d)$ is \emph{energy dominant} if $\Gamma\langle f\rangle \ll \Lambda$ for every $f\in \cF$.
An energy dominant measure $\Lambda$ is \emph{minimal energy dominant} if for every energy dominant $\nu$ we have $\Lambda \ll \nu$.

\begin{proposition}\label{prop:MinEnergyDom}
    Let $(X,d,\mu,\cE,\cF,\Gamma)$ be a $p$-Dirichlet space so that $(X,d)$ is compact. Then there exists a minimal energy dominant measure.
\end{proposition}
\begin{proof}
    Let $\cF(N) := \{f\in \cF : \|f\|_{\cF}\leq N\}$. By the separability of $L^p(X,\mu)$, which follows from the compactness of $X$, we can fix a countable set $S_N\subseteq \cF(N)$ that is dense in $\norm{\,\cdot \,}_{L^p(X,\mu)}$-norm.
    We enumerate the elements of $S_N$ as $\{f_{i,N}\}_{i\in \N}$. We also define $S=\bigcup_{N\in \N} S_N$ and
    \[
    \Lambda := \sum_{N\in \N} \sum_{i\in \N} \frac{1}{2^{i+N} N}\Gamma\langle f_{i,N}\rangle.
    \]
    We argue that $\Lambda$ is minimal energy dominant. Note that $\Lambda$ is a finite Radon measure by construction.
    Because $\Lambda$ is a countable sum of energy measures, $\Lambda \ll \nu$ when $\nu$ is any energy dominant measure. This proves the minimality of $\Lambda$.

    Let $f \in \mathcal{F}$. Then $f \in \mathcal{F}(N)$ for $N := \ceil{\norm{f}_{\mathcal{F}}}$, meaning there is $\{ h_i \}_{i = 1}^\infty \subseteq S_N$ such that $h_i \to f$ in $L^p(X,\mu)$. By the weak lower-semicontinuity, it holds for every Borel set $A \subseteq X$ that
    \[
        \Gamma\Span{f}(A) \leq \liminf_{i \to \infty} \Gamma\Span{h_i}(A).
    \]
    By construction, $\Gamma\Span{h_i} \ll \Lambda$ for each $i \in \N$. This and the previous display implies $\Gamma\Span{f} \ll \Lambda$.
\end{proof}

\subsection{Cartesian product of Dirichlet spaces}
If $(X,d_X,\mu_X)$ and $(X,d_Y,\mu_Y)$ are metric measure spaces, we denote the $\ell^\infty$-product metric by $d_X \times d_Y$, and the product measure by $\mu_X \otimes \mu_Y$.
We use the usual notation for the slices
\[
    \Omega^y := \{ x \in X : (x,y) \in \Omega \} \text{ and } \Omega_x := \{ y \in Y : (x,y) \in \Omega \},
\]
and also denote the sections
\[
    u(\,\cdot\,,y)(x) := u(x,y) \text{ and } u(x,\,\cdot\,)(y) := u(x,y).
\]

\begin{definition}\label{def:Product}
    Let $(X,d_X,\mu_X,\mathcal{E}_X,\mathcal{F}_X,\Gamma_X)$ and $(Y,d_Y,\mu_Y,\mathcal{E}_Y,\mathcal{F}_Y,\Gamma_Y)$ be two $p$-Dirichlet spaces. Their \emph{Cartesian product} is the collection
    \[
        (X \times Y,d_X \times d_Y,\mu_X \otimes\mu_Y, \mathcal{E},\mathcal{F},\Gamma)
    \]
    where the last three objects are given as follows.
    The domain $\mathcal{F}$ consists of those $u \in L^p(X \times Y, \mu_X \otimes \mu_Y)$ satisfying 
    \[
        u(\,\cdot\,,y) \in \mathcal{F}_X \text{ for $\mu_Y$-a.e. $y \in Y$} \text{ and }  u(x,\,\cdot\,) \in \mathcal{F}_Y \text{ for $\mu_X$-a.e. $x \in X$},
    \]
    and also
    \[
        \mathcal{E}(f) := \int_{X}\mathcal{E}_Y(u(x,\,\cdot \,))\, \text{d}\mu_X(x) + \int_{Y}\mathcal{E}_X(u(\,\cdot \,,y))\, \text{d}\mu_Y(y) < \infty.
    \]
    The measures $\Gamma\Span{\,\cdot\,}$ are given by
    \[
        \Gamma\Span{u}(\Omega) := \int_{X}\Gamma_Y\Span{u(x,\, \cdot\,)}(\Omega_x)\, \text{d}\mu_X(x) + \int_{Y}\Gamma_X\Span{u(\, \cdot\,,y)}(\Omega^y)\, \text{d}\mu_Y(y).
    \]
\end{definition}

In writing these definitions, it is not immediately obvious that the integrands are measurable. This will be shown in Lemma \ref{lemma:AbstractNonsense} below.
The objective of the section is to prove the following theorem.

\begin{theorem}\label{thm:Z}
Fix $p \in (1,\infty)$ and a scale function $\Psi$.
Let $(X,d_X,\mu_X,\mathcal{E}_X,\mathcal{F}_X,\Gamma_X)$ and  $(Y,d_Y,\mu_Y,\mathcal{E}_Y,\mathcal{F}_Y,\Gamma_Y)$ be $(p,\Psi)$-Poincar\'e--Dirichlet spaces.
Then their Cartesian product is also a $(p,\Psi)$-Poincar\'e--Dirichlet space.
\end{theorem}

The most technical part of Theorem \ref{thm:Z} is the following measurability result.

\begin{lemma}\label{lemma:AbstractNonsense}
Let $(Z,d,\mu,\mathcal{E},\mathcal{F},\Gamma)$ be the Cartesian product in Theorem \ref{thm:Z}.
Then, for every $u\in L^p(Z,\mu)$ and a Borel set $\Omega \subseteq Z$, the function
\[
y \mapsto
\begin{cases}
    \Gamma_X\langle u(\,\cdot\,,y)\rangle(\Omega^y) & \text{ if } u(\,\cdot\,,y) \in \mathcal{F}_X, \\
    \infty & \text{ otherwise},
\end{cases}
\]
is $\mu_Y$-measurable. Similarly,
\[
x \mapsto
\begin{cases}
    \Gamma_Y\langle u(x,\,\cdot\,)\rangle(\Omega_x) & \text{ if } u(x,\,\cdot\,) \in \mathcal{F}_Y, \\
    \infty & \text{ otherwise,}
\end{cases}
\]
is $\mu_X$-measurable.
\end{lemma}
\begin{proof}
    The two claims are symmetric, and thus it suffices to prove the first one.
    Consider the $\mu_Y$-measurable function $F: Y\to L^p(X,\mu_X)$ given by $F(y)=u(\,\cdot \, ,y)$.
    By the separability of $L^p(X,\mu)$ and Lusin's theorem, there exist closed subsets $K_i\subseteq Y$ for $i\in \N$ such that
    \[
    \mu_Y\left(Y\setminus \bigcup_{i\in \N} K_i\right)=0    
    \]
    and $F|_{K_i}$ is continuous.
    By the weak lower semicontinuity of the $p$-Dirichlet space, the function $K_i \to [0,\infty],\, y \mapsto \Gamma\Span{F(y)}(X)$, is lower-semicontinuous.
    Here we interpret $\Gamma_X\Span{F(y)}(X) = \infty$ if $F(y) \notin \mathcal{F}_X$.
    Hence, the set
    \[
        Y_M := \{ y \in K : \Gamma_X\Span{F(y)} \leq M \}
    \]
    is $\mu_Y$-measurable for all $M \geq 0$. 

    We first consider the case $\Omega = A \times B \subseteq Z$ where $A,B$ are Borel sets.
    Let $G : Y \to [0,\infty]$ be the function in the claim.
    Then $G|_{Y\setminus B} = 0$ and $G(y) = \Gamma_X\Span{F(y)}(A)$ for all $y \in B$.
    By the weak lower semicontinuity, $G|_{K_i \cap Y_M \cap B}$ is lower semicontinuous for every $i,M\in \N$. By definition, $G(y) = \infty$ in $B\setminus \bigcup_{M \in \N} Y_M$.
    By combining these facts, it follows that $G$ is $\mu_Y$-measurable.
    
    The measurability for a general Borel set $\Omega \subseteq Z$ follows from the monotone class theorem. 
\end{proof}

\begin{proof}[Proof of Theorem \ref{thm:Z}]
    The Cartesian product is denoted by $(Z,d,\mu,\mathcal{E},\mathcal{F},\Gamma)$.
    The integrals that determine $\mathcal{E}$ and $\Gamma$ are well-defined by Lemma \ref{lemma:AbstractNonsense}.
    Moreover, most of the required axioms are immediately transferred to the product through the definitions. Only the non-trivial ones  are checked here.

    We begin by proving the completeness of $\mathcal{F}$.
    Let $\{u_i\}_{i = 1}^\infty \subseteq \mathcal{F}$ be a Cauchy sequence, and let $u$ be the limit in $L^p(Z,\mu)$.
    Note that it is sufficient to prove that $u \in \mathcal{F}$ and that some subsequence $\{u_i\}_{i = 1}^\infty$ converges  in $\mathcal{F}$.
    Hence, we may assume by the Fubini's theorem that, for $\mu_X$-a.e. $x \in X$ and $\mu_Y$-a.e. $y \in Y$, we have convergence $u_i(x,\, \cdot\,) \to u(x,\, \cdot\,)$ and $u_i(\, \cdot\,,y) \to u(\, \cdot\,,y)$ in the respective $L^p$-spaces.
    Then, by the weak lower semicontinuity and Fatou's lemma,
    \begin{align*}
        \mathcal{E}(u) & = \int_{X} \mathcal{E}_Y(u(x,\, \cdot\,)) \, \text{d}\mu_X + \int_{Y} \mathcal{E}_X(u(\, \cdot\,,y)) \, \text{d}\mu_Y\\
        & \leq \int_{X} \liminf_{n \to \infty} \mathcal{E}_Y(u_i(x,\, \cdot\,)) \, \text{d}\mu_X + \int_{Y} \liminf_{n \to \infty} \mathcal{E}_X(u_i(\, \cdot\,,y)) \, \text{d}\mu_Y\\
        & \leq \liminf_{n \to \infty} \int_{X} \mathcal{E}_Y(u_i(x,\, \cdot\,)) \, \text{d}\mu_X + \liminf_{n \to \infty} \int_{Y} \mathcal{E}_X(u_i(\, \cdot\,,y)) \, \text{d}\mu_Y\\
        & \leq \liminf_{i \in \N} \mathcal{E}(u_i) < \infty.
    \end{align*}
    The finiteness of the last row follows from the fact that $\{u_i\}_{i = 1}^\infty \subseteq \mathcal{F}$ is Cauchy. Hence $u \in \mathcal{F}$. By repeating the previous arguments, we also get
    \begin{align*}
        \limsup_{i \to \infty} \mathcal{E}(u - u_i) & \leq \limsup_{i \to \infty} \left( \liminf_{ j \to \infty } \mathcal{E}(u_j - u_i) \right) = 0.
    \end{align*}
    This completes the proof of the completeness of $\mathcal{F}$. A similar treatment yields the weak lower semicontinuity.

    
    The doubling property of $\mu$ is obvious from the doubling properties of the marginals. Hence \textrm{VD} holds.
    The upper capacity estimate $\cCap(\Psi)$ follows by taking functions of the form
    \[
        (x,y) \mapsto \phi(x) \cdot \psi(y)
    \]
    where $\phi \in \mathcal{F}_X \cap C(X)$ and $\psi \in \mathcal{F}_Y \cap C(Y)$ are provided by $\cCap(\Psi)$ of the factors. Lastly, the Poincar\'e inequality $\PI(\Psi)$ can be verified using the same argument as in \cite[Proof of Theorem 3]{BBtensor}.
    We emphasize that it is essential for $\Psi$ to be independent of the spatial variable, and to be the same in both factors.
\end{proof}

Recall minimal energy dominant measure from Subsection \ref{subsec:BasicProp}.

\begin{corollary}\label{cor:Lambda-minimal}
    Assume the hypotheses in Theorem \ref{thm:Z}, and that $(X,d_X)$, $(Y,d_Y)$ are compact. Let $(Z,d,\mu,\mathcal{E},\mathcal{F},\Gamma)$ be the Cartesian product in Theorem \ref{thm:Z}, and $\Lambda_X$ and $\Lambda_Y$ be a pair of minimal energy dominant measures of $(X,d_X,\mu_X,\mathcal{E}_X,\mathcal{F}_X,\Gamma_X)$ and  $(Y,d_Y,\mu_Y,\mathcal{E}_Y,\mathcal{F}_Y,\Gamma_Y)$, respectively. Then
    \[
        \Lambda := \Lambda_X \otimes \mu_Y + \mu_X \otimes \Lambda_Y.
    \]
    is a minimal energy dominant measure of $(Z,d,\mu,\mathcal{E},\mathcal{F},\Gamma)$.
\end{corollary}

\begin{proof}
    The fact that $\Lambda$ is energy dominant is obvious from the definition of the measures $\Gamma$ in Definition \ref{def:Product}.
    Moreover, since any pair of minimal energy dominant measures are mutually absolutely continuous by definition, it is sufficient to prove the claim for any choice of $\Lambda_X$ and $\Lambda_Y$. Hence, without loss of generality we assume that $\Lambda_X$ and $\Lambda_Y$ are finite Radon measures of the form
    \[
        \Lambda_X = \sum_{i = 1}^\infty \alpha_i \Gamma_X\Span{f_i} \text{ and } \Lambda_Y = \sum_{i = 1}^\infty \beta_i \Gamma_Y\Span{g_i}.
    \]
    Such choices exist by the proof of Proposition \ref{prop:MinEnergyDom}.
    
    We verify the minimality of $\Lambda$. Let $\nu$ be another energy dominant measure.
    Consider the functions of the form $u(x,y) := f(x)$ and $v(x,y) := g(y)$, where $f \in \mathcal{F}_X$ and $g \in \mathcal{F}_Y$. By the compactness of the ambient spaces, $u,v \in \mathcal{F}$.
    A direct computation shows that
    \[
        \Gamma\Span{u} = \Gamma_X \Span{f} \otimes \mu_Y \text{ and } \Gamma\Span{v} = \mu_X \otimes \Gamma_X \Span{g}.
    \]
    Both of these measures are absolutely continuous with respect to $\nu$. Since $\Lambda$ is a sum of such measures, it is absolutely continuous with respect to $\nu$.
\end{proof}

\section{Conformal Dimension}\label{Sec:Confdim}

This section relates properties of $p$-Dirichlet spaces to the conformal dimension of the ambient space.
We consider the following general setting, and impose further assumptions when necessary.

\begin{assumption}\label{Assumption:Confdim}
    Fix $p \in (1,\infty)$, $\beta > 0$, $d_{\textrm{H}} > 0$, and let $(X,d,\mu,\mathcal{E},\mathcal{F},\Gamma)$ be a $(p,\beta)$-Poincar\'e--Dirichlet space where $(X,d)$ is $d_{\textrm{H}}$-Ahlfors regular and $\mu$ is the $d_{\textrm{H}}$-Hausdorff measure.
    The underlying metric space $(X,d)$ is also assumed to be of bounded turning and compact.
    
\end{assumption}

For each $\varepsilon > 0$ we fix a collection $\mathcal{U}_\varepsilon$ consisting of open balls in $(X,d)$ so that the centers form an $\varepsilon$-net. We also let $\{ \psi_{U} \}_{U \in \mathcal{U}_\varepsilon}$ be a partition of unity subordinate to $\mathcal{U}_\varepsilon$ as in Lemma \ref{lemma:POU}.

\subsection{Moduli estimates}
We begin by providing estimates on the discrete moduli. Proof methods are very similar to those in Heinonen and Koskela \cite{HK}.
No further conditions are imposed here.

The first estimates concern annular moduli.

\begin{lemma}\label{lemma:Carrasco}
    Suppose that Assumption \ref{Assumption:Confdim} holds.
    There are $C,R > 0$ such that, for all $x \in X$ and $0 < r \leq R$, we have the following. If $\mathcal{C}$ denotes the family of curves in $X$ that intersect with both $B(x,r)$ and $X\setminus B(x,2r)$, then
    \[
        C^{-1} \left(\frac{\varepsilon}{r}\right) ^{\beta - d_{\Hr}} \leq \Mod_p(\mathcal{C},\mathcal{U}_\varepsilon) \leq C \left(\frac{\varepsilon}{r}\right) ^{\beta - d_{\Hr}}
    \]
    for all $0 < \varepsilon \leq r$.
\end{lemma}

\begin{proof}
Fix $R:= \diam(X)/32$.
Let $x \in X$, $0 < r \leq R$ and $0 < \varepsilon \leq r$. Throughout the proof, $\mathcal{C}$ is as in the claim.
Note that $\mathcal{C}$ is non-empty because $X$ is assumed to be of bounded turning.
By \cite[Proposition 2.2]{BourK}, we may prove the estimates for $\Mod_p(\mathcal{C},\mathcal{U}_{\varepsilon/32})$ instead of $\Mod_p(\mathcal{C},\mathcal{U}_{\varepsilon})$.

The latter inequality in the claim is proven first.
By \eqref{eq:AR} and $\cCap(\beta)$, we can construct 
$\varphi \in \mathcal{F}$ with $\varphi|_{B(x,5r/4)} = 1$, $\varphi|_{X \setminus B(x,7r/4)}=0$ and
\[
\mathcal{E}(\varphi)\lesssim r^{d_{\Hr}-\beta}.
\]
More precisely, we cover $B(x,5r/4)$ with $\{ B(x_i, r/8) \}_{i \in I}$ where $\{x_i\}_{i \in I} \subseteq B(x,5r/4)$ is an $r/8$-net. Then, for each $i \in I$, we take $\varphi_i \in \mathcal{F}$ provided by $\cCap(\beta)$, and define
\[
    \varphi := \min\left\{\sum_{i \in I} \varphi_i,1\right\}. 
\]
The energy estimate for $\varphi$ follows from the Lipschitz contractivity of the $p$-Dirichlet space because we can bound the cardinality of $I$ with a number depending only on the constants in \eqref{eq:AR}.

We now define a function $f :\mathcal{U}_{\varepsilon/32} \to [0,\infty)$ given by
\[
f(U) := \kint_U \varphi\, \text{d}\mu.
\]
By $\varepsilon \leq r$ and $\diam(U)\leq \varepsilon/16$ for all $U\in \mathcal{U}_{\varepsilon/32}$, we have $U\subseteq B(x,5/4r)$ whenever $U \cap B(x,r) \neq \emptyset$. Similarly, if $U \in \mathcal{U}_{\varepsilon/32}$ satisfies $U \cap X\setminus B(x,2r)\neq \emptyset$, we have $U \subseteq X\setminus B(x,7/4r)$.
Hence we have
\begin{equation}\label{eq:ForMod}
   f(U) = \begin{cases}
    1 & \text{ if } U \cap B(x,r) \neq \emptyset \\
    0 & \text{ if } U \cap X \setminus B(x,2r) \neq \emptyset.
    \end{cases}
\end{equation}
We now define $\rho : \mathcal{U}_{\varepsilon/32} \to [0,\infty)$ by
\[
\rho(U) := \max_{\substack{ V \in \mathcal{U}_{\varepsilon/32} \\ V \cap U \neq \emptyset}} \abs{f(U) - f(V)}.
\]
It follows from the triangle inequality and \eqref{eq:ForMod} that $\rho$ is $\mathcal{C}$-admissible. Its $p$-mass can be estimated
\begin{align*}
    \sum_{U \in \mathcal{U}_{\varepsilon/32}} \rho(U)^p & \lesssim \sum_{\substack{U,V \in \mathcal{U}_{\varepsilon/32} \\ U\cap V \neq \emptyset }} \abs{f(U) - f(V)}^p \lesssim \varepsilon^{\beta - d_{\Hr}} \mathcal{E}_p(\varphi) \lesssim \left(\frac{\varepsilon}{r}\right)^{\beta - d_{\Hr}}.
\end{align*}
The first inequality holds because we have bounded overlapping. The second follows from $\PI(\beta)$ and the bounded overlapping, and the the third is a consequence of the energy upper bound of $\varphi$.
Since $\rho$ is $\mathcal{C}$-admissible, this proves the latter inequality in the claim.

The converse inequality is considered next.
Let $\rho : \mathcal{U}_{\varepsilon/32} \to [0,\infty)$ be $\mathcal{C}$-admissible.
Since we aim to prove a lower bound for the $p$-mass of $\rho$, we may redefine $\rho$ as follows. Let $\tilde{\rho}$ be such that $\tilde{\rho}(U) = 0$ if $2U \cap (X\setminus B(x,4r)) \neq \emptyset$ and $\tilde{\rho}(U) = \rho(U)$ otherwise.
Because $\tilde{\rho}$ is admissible with smaller $p$-mass, by replacing $\rho$ with $\tilde{\rho}$ if necessary, we may assume that $\rho(U) = 0$ whenever $2U \cap (X\setminus B(x,4r)) \neq \emptyset$.

We define a function $f:\mathcal{U}_{\varepsilon/32} \to [0,\infty)$ according to
\begin{equation*}
    f(U) := \min_{ \gamma } \sum_{ \substack{U \in \mathcal{U}_{\varepsilon/32}\\ U \cap \gamma \neq \emptyset } }\rho(U)
\end{equation*}
where the minimum is taken over all curves $\gamma$ such that $\gamma$ intersects with both $U$ and $X \setminus B(x,2r)$.
Since $\rho$ is $\mathcal{C}$-admissible, $f(U) \geq 1$ for all $U$ with $2U\cap B(x,r/2)\neq \emptyset$.
We also clearly have $f(U) = 0$ for all $U \in \mathcal{U}_{\varepsilon/32}$ with $2U \cap X\setminus B(x,4r) \neq \emptyset$ because we can just take $\gamma$ to be a constant sequence.
Now, we define a function $u$ using the partition of unity given by
\begin{equation}\label{eq:rho->u}
    u := \sum_{U \in \mathcal{U}_{\varepsilon/32}} f(U) \cdot \psi_{U}.
\end{equation}
Since $X$ is assumed to be compact, $u \in \mathcal{F}$ because it is a finite sum of functions in $\mathcal{F}$.
By the properties of $\{\psi_{U}\}_{U \in \mathcal{U}_{\varepsilon/32}}$ and $f$, we have $u \geq 1$ on $B(x,r/2)$ and $u = 0$ on $X \setminus B(x,4r)$.
Now, if $y \in X$ so that $d(x,y) = 5r$, we can estimate using $\PI(\beta)$ and $d_{\textrm{H}}$-the Ahlfors regularity that
\begin{equation}\label{eq:Annular-1}
    1 \leq \abs{ u_{B(x,r/2)} - u_{B(y,r/2)} }^p \lesssim \kint_{B(x,10r)} \abs{u - u_{B(x,10r)}}^p \, \text{d}\mu \lesssim r^{\beta - d_{\textrm{H}}} \mathcal{E}(u).
\end{equation}

Then we estimate $\mathcal{E}(u)$ using the properties of the partition of unity.
For each $U \in \mathcal{U}_{\varepsilon/32}$ and $x \in U$ we have
\[
    u(x) = \sum_{ \substack{V \in \mathcal{U}_{\varepsilon/32} \\ U \cap 2V \neq \emptyset} } \left(f(V) - f(U)\right) \psi_V(x)  + f(U).
\]
By using the strong locality and the triangle inequality of $\Gamma\Span{\, \cdot\, }$, we get
\[
    \Gamma\Span{u}(U)^{\frac{1}{p}} \leq  \sum_{\substack{V \in \mathcal{U}_{\varepsilon/32} \\ U \cap 2V \neq \emptyset}} \abs{f(V) - f(U)} \mathcal{E}(\psi_V)^{\frac{1}{p}}   \lesssim  \sum_{\substack{V \in \mathcal{U}_{\varepsilon/32} \\ U \cap 2V \neq \emptyset}} \abs{f(V) - f(U)} \varepsilon^{(d_{\textrm{H}} - \beta)/p} .
\]
Note that the number of $V \in \mathcal{U}_{\varepsilon/32}$ with $U \cap 2V \neq \emptyset$ can be bounded by $N \in \N$ depending only the Ahlfors regularity constants in \eqref{eq:AR}.
The Hölder's inequality now gives
\[
    \sum_{\substack{V \in \mathcal{U}_{\varepsilon/32} \\ U \cap 2V \neq \emptyset}} \abs{f(V) - f(U)} \varepsilon^{(d_{\textrm{H}} - \beta)/p} \leq N^{\frac{p-1}{p}} \left(\sum_{\substack{V \in \mathcal{U}_{\varepsilon/32} \\ U \cap 2V \neq \emptyset}} \abs{f(V) - f(U)}^p \varepsilon^{d_{\textrm{H}} - \beta} \right)^{\frac{1}{p}}.
\]
By raising the previous two inequalities to the power $p$ and summing them over all $U \in \mathcal{U}_{\varepsilon/32}$, we get
\begin{equation}\label{eq:Annular-2}
    \mathcal{E}(u) \lesssim \varepsilon^{d_{\Hr} - \beta} \sum_{\substack{U,V \in \mathcal{U}_{\varepsilon/32} \\ U\cap 2V \neq \emptyset }}  \abs{f(U) - f(V)}^p.
\end{equation}

Lastly, we estimate the sum in the right-hand side.
Let $U,V \in \mathcal{U}_{\varepsilon/32}$ such that $2U \cap 2V \neq \emptyset$, and assume without loss of generality that $f(U) \leq f(V)$.
Take a curve $\gamma_U$ that minimizes the sum in the definition of $f(U)$. By the bounded turning condition, there is a curve $\gamma$ connecting $V$ to a point in $\gamma_U \cap U$ with $\diam(\gamma) \leq 4K\varepsilon$. Since $\gamma \cup \gamma_U$ is a competitor for $f(V)$, we have
\begin{align}\label{eq:Annular-3}
    \abs{f(U) - f(V)} & =  f(V) - f(U) \leq \sum_{ \substack{ W \in \mathcal{U}_{\varepsilon/32} \\W \cap (\gamma \cup \gamma_U) \neq \emptyset }} \rho(U) - \sum_{ \substack{ W \in \mathcal{U}_{\varepsilon/32} \\ W \cap \gamma_U \neq \emptyset} } \rho(W)\\
    & \leq \sum_{\substack{ W \in \mathcal{U}_{\varepsilon/32} \\ W \cap \gamma \neq \emptyset}} \rho(W) \leq \sum_{ \substack{ W \in \mathcal{U}_{\varepsilon/32} \\ 8KU \cap W \neq \emptyset } } \rho(W).\nonumber
\end{align}
We can bound the number of $W \in \mathcal{U}_{\varepsilon/32}$ in the last sum by a number depending only on $K$ in the bounded turning condition and the Ahlfors regularity constants.
By summing the previous inequality over $U,V$, we get
\[
    \sum_{\substack{U,V \in \mathcal{U}_{\varepsilon/32} \\ U\cap 2V \neq \emptyset }}  \abs{f(U) - f(V)}^p \leq \sum_{\substack{U,V \in \mathcal{U}_{\varepsilon/32} \\ 2U\cap 2V \neq \emptyset }}  \abs{f(U) - f(V)}^p \lesssim \sum_{U \in \mathcal{U}_{\varepsilon/32}} \rho(U)^p.
\]
By combining this with \eqref{eq:Annular-1} and \eqref{eq:Annular-2}, and the fact that $\rho$ is an arbitrary $\mathcal{C}$-admissible function, we obtain the remaining inequality of the claim.
\end{proof}

Then we verify the ball-Loewner type estimate. A continuous version of it first appeared in Bonk and Kleiner \cite{BK05}, and see also \cite{BourK,MuruganShimizu} for discrete variants.

\begin{lemma}\label{lemma:CBL}
    Suppose that Assumption \ref{Assumption:Confdim} holds.
    For every $A > 0$ there are $M,L \geq 1$ such that the following holds. If $x,y \in X$ and $r > 0$ such that $B(x,r)$ and $B(y,r)$ are disjoint open balls with $\dist(B(x,r),\dist(B(y,r))) \leq Ar$, and if $\mathcal{C}$ is the set of curves that intersect with both $B(x,r)$ and $B(y,r)$ with diameter at most $Lr$, then
    \[
        \Mod_p(\mathcal{C},\mathcal{U}_\varepsilon) \geq M^{-1} \left(\frac{\varepsilon}{r}\right) ^{\beta - d_{\Hr}}
    \]
    for all $0 < \varepsilon \leq r$.
\end{lemma}

\begin{proof}
    Let $\sigma,K$ be the constants in $\PI(\beta)$ and the bounded turning condition, respectively. Let $x,y \in X$, $r > 0$ and $A$ be as in the claim, $L := 2\sigma(A+10)$ and $\mathcal{C}$ be as in the claim. By \cite[Proposition 2.2]{BourK}, we may prove the estimates for $\Mod_p(\mathcal{C},\mathcal{U}_{\varepsilon/(16K)})$ instead of $\Mod_p(\mathcal{C},\mathcal{U}_{\varepsilon})$.
    In other words, given any $\mathcal{C}$-admissible $\rho : \mathcal{U}_{\varepsilon/(16K)} \to [0,\infty)$, we prove a lower bound for the $p$-mass
    \[
        \sum_{U\in \mathcal{U}_{\varepsilon/(16K)}} \rho (U)^p \geq M^{-1}\left(\frac{\varepsilon}{r}\right) ^{\beta - d_{\Hr}}.
    \]

    Let such $\rho$ be given.
    By a similar redefinition argument as in the proof of Lemma \ref{lemma:Carrasco}, we may assume without loss os generality that $\rho(U) = 0$ for all $U$ with $2U \cap B(x,3r/4) \neq \emptyset$.
    Define $\mathcal{C}_b$ be the collection of ``bad curves'', consisting of those that intersect with both $B(y,\sigma(A+10)r)$ and $X\setminus B(y, 2\sigma(A+10)r)$. Take any $\mathcal{C}_b$-admissible $\rho_b : \mathcal{U}_{\varepsilon/(16K)} \to [0,\infty)$ such that $\rho_b(U) = 0$ for every $U$ with $U \cap B(y,\sigma(A+10)r) \neq \emptyset$. Consider $\tilde{\rho} := \max \{ \rho,\rho_b \} $, and define $f : \mathcal{U}_{\varepsilon/(16/K)} \to [0,\infty)$ be given by
    \[
        f(U) := \min_{\gamma} \sum_{ \substack{ V \in \mathcal{U}_{\varepsilon/(16K)} \\
        U \cap \gamma \neq \emptyset} } \tilde{\rho}(V)
    \]
    where the minumum is taken over all curves $\gamma$ that intersect with $U$ and $B(x,r) \cup X \setminus B(y,2\sigma(A+10)r)$. Note that $f$ satisfies
    \[
        \begin{cases}
        f(U) = 0 \text{ if } 2U \cap B(x,r/2) \neq \emptyset,\\
        f(U) \geq 1 \text{ if } 2U \cap B(y,r/2) \neq \emptyset.
        \end{cases}
    \]
    The first row holds from $\rho(U) = 0 = \rho_b(U)$ when $2U \cap B(x,3r/4)$ by taking any constant curve $\gamma = \{z\} \subseteq U$. The second row follows from $\tilde{\rho} \geq \rho,\rho_b$ by noting that any $\gamma$ which is a competitor for $f(U)$ contains a subcurve in $\mathcal{C} \cup \mathcal{C}_b$.

    Now, we use the partition of unity to construct a function
    \[
        u := \sum_{U \in \mathcal{U}_{\varepsilon/(16K)} } f(U) \cdot \psi_U.
    \]
    By a similar computation as in \eqref{eq:Annular-2}, we have
    \[
        \Gamma\Span{u}( B(y,\sigma(A+4)r) ) \lesssim \sum_{ \substack{U \in \mathcal{U}_{\varepsilon/(16K)} \\ 2U \cap B(y,\sigma(A+4)r) } } \sum_{ \substack{V \in \mathcal{U}_{\varepsilon/(16K)} \\ 2U \cap 2V \neq \emptyset }} \abs{f(U) - f(V)}^p\varepsilon^{d_{\textrm{H}} - \beta}.
    \]
    If $U,V$ are as above, then
    \[
        \abs{f(U) - f(V)}^p \lesssim \sum_{ \substack{ W \in \mathcal{U}_{\varepsilon/(16K)} \\ 8KU \cap W \neq \emptyset } } \tilde{\rho}(W)^p = \sum_{ \substack{ W \in \mathcal{U}_{\varepsilon/(16K)} \\ 8KU \cap W \neq \emptyset } } \rho(W)^p.
    \]
    The inequality follows from the same argument as in \eqref{eq:Annular-3}, and the equality from $\rho_b(W) = 0$ whenever $W$ is as above.
    By the bounded overlapping, the previous two displays yield
    \[
        \Gamma\Span{u}( B(y,\sigma(A+4)r) ) \lesssim \varepsilon^{d_{\textrm{H}} - \beta} \sum_{ U \in \mathcal{U}_{\varepsilon/(16K)} } \rho(U)^p.
    \]
    The left-hand side in the previous inequality is estimated using $\PI(\beta)$
    \[
        \int_{B(y,(A+4)r)} \abs{u - u_{B(y,(A+4)r)}}^p \, \text{d}\mu \lesssim r^\beta \Gamma\Span{u}( B(y,\sigma(A+4)r) ).
    \]
    By the properties of $f$ and the definition of $u$, we have $u = 0$ on $B(x,r/2)$ and $u = 1$ on $B(y,r/2)$. This gives
    \[
        r^{d_{\textrm{H}}} \lesssim \frac{\min \{ \mu(B(x,r/2)), \mu(B(y,r/2)) \} }{2^p} \lesssim \int_{B(y,(A+4)r)} \abs{u - u_{B(y,(A+4)r)}}^p \, \text{d}\mu.
    \]
    The combination of the previous three displays yields the desired lower-bound for the $p$-mass of $\rho$, which completes the proof.
\end{proof}

The combination of Lemma \ref{lemma:Carrasco} and Lemma \ref{lemma:CBL} yields an analytic description of the conformal dimension in terms of Dirichlet spaces.

\begin{corollary}\label{cor:confdim}
    Suppose that Assumption \ref{Assumption:Confdim} holds.
    If additionally $d_{\textup{H}} = \beta$ then the conformal dimension $Q$ of $(X,d)$ is equal to $p$.
\end{corollary}

\begin{proof}
    By Lemma \ref{lemma:Carrasco} and Lemma \ref{lemma:CBL}, the metric space $(X,d)$ is Combinatorially $Q$-Loewner in the sense of Bourdon--Kleiner \cite{BourK}. The claim now follows from a characterization of the conformal dimension due to Carrasco-Piaggio \cite[Theorem 1.3]{Carrasco}. Alternatively, see \cite[Lemma 4.2]{EBConf} for a direct statement.
 \end{proof}

\subsection{Auxiliary Sobolev spaces}\label{subsec:Aux}
The remainder of the section works in the setting of Assumption \ref{Assumption:Confdim} under the additional condition $d_\textrm{H} = \beta$. By Corollary \ref{cor:confdim}, $p$ is then equal to the conformal dimension of $(X,d)$. In order to highlight this, we replace $p$ with $Q$.

Suppose that Assumption \ref{Assumption:Confdim} and $d_{\textrm{H}}=\beta$ hold.
A collection $(\rho,\nu,\Xi,\mathcal{W})$ is called an \emph{auxiliary $Q$-Sobolev space} if the following hold.
\begin{enumerate}
    \item $\rho$ is a metric on $X$ so that it is quasisymmetric to $d$.
    \item $\nu$ is a doubling measure on $(X,\rho)$.
    \item $\Xi : L^Q(X,\nu) \to [0,\infty]$ is given by
    \[
        \Xi(f) :=  \sup_{r > 0}  
    \int_{X} \kint_{B_\rho(x,r)} \frac{\abs{f(x) - f(y)}^Q}{\nu(B_\rho(x,r))} \text{d}\nu(y)\,\text{d}\nu(x).
    \]
    \item $\mathcal{W} = L^Q(X,\nu) \cap \Xi^{-1}([0,\infty))$ and it is given the norm
    \[
  \norm{f}_{\mathcal{W}} := \norm{f}_{L^Q(X,\nu)} + \Xi(f)^{\frac{1}{Q}}.  
\]
\end{enumerate}
The $Q$-energy $\Xi$ can be understood as a variant of a Korevaar--Schoen energy. See \cite{ShimizuKS,Baudoin24} and therein references.
The results of the section show that auxiliary $Q$-Sobolev spaces are all equivalent to the each other.
We call this a \emph{``conformal invariance''} because doubling measures are invariant under quasisymmetric homeomorphisms.
The important special case where the attainment is realized is also studied.
See \cite[Section 7]{HK} and references therein for related results.
Our methods are influenced by a work of Murugan and Shimizu \cite[Section 9]{MuruganShimizu} where a similar result on the Sierpi\'nski carpet is established.

Let us remark that we often simultaneously work with two metrics $d$ and $\rho$, and two measures $\mu$ and $\nu$.
The intended one in each instance will be clear from the notation. For instance, we denote the open balls $B_d(x,r)$ and $B_\rho(x,r)$ and the integral averages $f_{A,\mu}$ and $f_{A,\nu}$.

\subsection{Conformal invariance}
In the following proposition, we only consider continuous functions to avoid the technical detail that the reference measures are different.

\begin{proposition}\label{prop:Conf-Inv}
    Suppose that Assumption \ref{Assumption:Confdim} holds, and assume $d_{\textup{H}}=\beta$.
    Let $(\rho,\nu,\Xi,\mathcal{W})$ be an auxiliary $Q$-Sobolev space. Then $\mathcal{F} \cap C(X) = \mathcal{W} \cap C(X)$. Moreover, there is a constant $C \geq 1$ such that, for all $f \in \mathcal{F} \cap C(X)$, we have
    \[
       C^{-1}\norm{f}_{\mathcal{W}} \leq \norm{f}_{\mathcal{F}} \leq C\norm{f}_{\mathcal{W}}.
    \]
\end{proposition}

A key ingredient in the proof of Proposition \ref{prop:Conf-Inv} is a two-measure Poincar\'e inequality. In proofs below, we emphasize the reference measure by writing $f_{A,\mu}$ when computing averages.

\begin{lemma}\label{lemma:TwoMeasure}
    Assume the hypotheses of Proposition \ref{prop:Conf-Inv}.
    There are $C,\kappa \geq 1$ such that, for all $x \in X$, $s > 0$ and $f \in \mathcal{F} \cap C(X)$, we have
    \[
         \int_{B_\rho(x,s)} \abs{f - f_{B_\rho(x,s),\mu}}^Q \, \textup{d}\nu \leq C \nu(B_\rho(x,s)) \int_{B_\rho(x,\kappa s)} \, \textup{d}\Gamma\Span{f}.
    \]
\end{lemma}

\begin{proof}
    It follows from \cite[Corollary 5.5]{Anttila} that, for all $x \in X$, $s > 0$ and $f \in \mathcal{F} \cap C(X)$, we have
    \begin{equation}\label{eq:TwoMeasure}
        \int_{B_d(x,r)} \abs{f - f_{B_d(x,r),\mu}}^Q \, \textup{d}\nu \leq C \nu(B_d(x,r)) \int_{B_d(x,\sigma r)} \, \text{d}\Gamma\Span{f}.
    \end{equation}
    The constants $C,\sigma$ are independent of $x \in X$, $s > 0$ and $f \in \mathcal{F}$. The hypotheses of \cite[Corollary 5.5]{Anttila} follow from Assumption \ref{Assumption:Confdim} and $d_{\textrm{H}}=\beta$.
    This is not yet the desired two-measure Poincar\'e inequality because the metric is not the correct one.
    To fix this, we use the fact that quasisymmetries almost preserves open balls.

    Suppose that $d$ is $\eta$-quasisymmetric to $\rho$, and let $x \in X$, $s > 0$ and $L \geq 1$.
    We first show that
    \begin{equation}\label{eq:QS-balls}
        B_\rho(x,s) \subseteq B_d(x,2r) \text{ and } B_d(x,2L r) \subseteq B_\rho(x,\kappa s)
    \end{equation}
    where
    \[
        r := \sup_{y \in B_\rho(x,s)} d(x,y) \text{ and } \kappa := \frac{1}{\eta^{-1}\left( \frac{1}{4L} \right)}.
    \]
    The first inclusion in \eqref{eq:QS-balls} is obvious.
    Let $y \in B_d(x,2L r)$ and take any $z \in B_\rho(x,s)$ so that $d(x,z) \geq r/2$.
    By using the fact that $d$ is $\eta$-quasisymmetric to $\rho$ we have
    \[
        \frac{1}{4L} \leq \frac{d(x,z)}{d(x,y)} \leq \eta\left( \frac{\rho(x,z)}{\rho(x,y)} \right) < \eta\left( \frac{s}{\rho(x,y)} \right).
    \]
    Since $\eta$ is an increasing homeomorphism, we get $\rho(x,y) \leq \kappa s$, and hence the second inclusion in \eqref{eq:QS-balls} holds.

    Let $x,s$ and $f$ be as in the claim, and $r,\kappa$ as in \eqref{eq:QS-balls} for $L = \sigma$. We have
    \begin{align*}
        & \quad \, \int_{B_\rho(x,s)} \abs{f - f_{B_\rho(x,s),\mu}}^Q \, \text{d}\nu \leq \int_{B_d(x,2r)} \abs{f - f_{B_\rho(x,s),\mu}}^Q \, \text{d}\nu\\
        \lesssim & \quad \int_{B_d(x,2r)} \abs{f - f_{B_d(x,2r),\mu}}^Q \, \text{d}\nu + \nu(B_\rho(x,s)) \kint_{B_\rho(x,s)} \abs{f - f_{B_d(x,2r),\mu}}^Q\, \text{d}\mu\\
         \lesssim & \quad \int_{B_d(x,2r)} \abs{f - f_{B_d(x,2r),\mu}}^Q \, \text{d}\nu + \nu(B_\rho(x,s)) \kint_{B_d(x,2r)} \abs{f - f_{B_d(x,2r),\mu}}^Q\, \text{d}\mu.
    \end{align*}
    We use \eqref{eq:QS-balls} in the first row, $\abs{a+b}^Q \leq 2^{Q-1}(\abs{a} + \abs{b})$ and Jensen's inequality in the second row, and the doubling property of $\nu$ and \eqref{eq:QS-balls} in the third.
    The two terms on the last row are estimated separately.
    By applying \eqref{eq:TwoMeasure} to the first term and $\PI(\beta)$ to the second term, both of them are bounded by
    \[
        \nu(B_\rho(x,s)) \int_{B_d(x,2\sigma r)} \text{d}\Gamma\Span{f} \leq \nu(B_\rho(x,s)) \int_{B_\rho(x,\kappa s)} \text{d}\Gamma\Span{f}.
    \]
    When estimating the second term, we use the Ahlfors regularity and $d_{\textrm{H}}=\beta$.
    This concludes the proof.
\end{proof}

\begin{proof}[Proof of Proposition \ref{prop:Conf-Inv}]
    The equality of the function spaces is verified first. Let $f \in \mathcal{F} \cap C(X)$ and $r > 0$, and let $\mathcal{N}_r \subseteq X$ be an $r$-net in the metric $\rho$. By the two-measure Poincar\'e inequality in Lemma \ref{lemma:TwoMeasure} and the doubling property of $\nu$,
    \begin{align*}
        & \quad \int_X \kint_{B_\rho(x,r)} \frac{\abs{f(x) - f(y)}^Q}{\nu(B_\rho(x,r))} \,\text{d}\nu(y) \, \text{d}\nu(x)\\
        \lesssim & \quad \sum_{z \in \mathcal{N}_r} \int_{B_\rho(z,2r)} \kint_{B_\rho(z,2r)} \frac{\abs{f(x) - f(y)}^Q}{\nu(B_\rho(z,r))} \, \text{d}\nu(y)\, \text{d}\nu(x)\\
        \lesssim & \quad \sum_{z \in \mathcal{N}_r} \int_{B_\rho(z,2r)} \abs{f - f_{B_{\rho}(z,2r),\mu}}^Q \, \text{d}\nu \lesssim \sum_{z \in \mathcal{N}_r} \int_{B_\rho(z,2\lambda r)} \, \text{d}\Gamma\Span{f} \lesssim \mathcal{E}(f).
    \end{align*}
    We get $\Xi(f) \lesssim \mathcal{E}(f) < \infty$ by taking the supremum over $r > 0$, which yields the first inclusion $\mathcal{F} \cap C(X) \subseteq \mathcal{W} \cap C(X)$.

    Then assume $f \in \mathcal{W} \cap C(X)$. For $r > 0$ let $\mathcal{N}_r$ be as above. By combining Lemma \ref{lemma:POU} with \eqref{eq:QS-balls} we can construct a partition of unity
    $\{ \psi_V \}_{V \in \mathcal{V}_r}$ subordinate to $\mathcal{V}_r := \{ B_\rho(z,r) \}_{z \in \mathcal{N}_r}$ satisfying the following.
    \begin{enumerate}
        \vspace{2pt}
        \item $\sum_{V \in \mathcal{V}_r} \psi_{V}(x) = 1$ and $0 \leq \psi_{V}(x)\leq 1$ for all $V \in \mathcal{V}_r$ and $x \in X$.
        \vspace{2pt}
        \item It holds for all $V \in \mathcal{V}_r$ that $\psi_{V}|_{X \setminus 2V} = 0$.
        \vspace{2pt}
        \item There is $C \geq 1$ such that, for every $V \in \mathcal{V}_r$, we have $\mathcal{E}(\psi_V) \leq C$.
    \end{enumerate}
    Note that the last condition is a consequence of $d_{\textrm{H}} = \beta$.
    Now, consider the approximations
    \[
        f_r := \sum_{V \in \mathcal{V}_r} \left( \kint_{V} f\, \text{d}\nu \right) \psi_V.
    \]
    By using a similar computation as in \eqref{eq:Annular-2}, their energies can be estimated by
    \begin{align*}
        \mathcal{E}(f_r)
        & \lesssim \sum_{V \in \mathcal{V}_\varepsilon} \mathcal{E}(\psi_V)  \int_{8V} \kint_{8V} \frac{\abs{f(x) - f(y)}^Q}{\nu(V)} \, \text{d}\nu(x)\text{d}\nu(y)\\
        & \lesssim \sum_{V \in \mathcal{V}_\varepsilon} \int_{8V} \kint_{B_\rho(x,16r)} \frac{\abs{f(x) - f(y)}^Q}{\nu(B_\rho(x,16r))} \, \text{d}\nu(x)\text{d}\nu(y)\\
        & \lesssim \int_X \kint_{B_\rho(x,16r)} \frac{\abs{f(x) - f(y)}^Q}{\nu(B_\rho(x,16r))} \, \text{d}\nu(x)\text{d}\nu(y) \leq \Xi(f).
    \end{align*}
    In the second row we used $\mathcal{E}(\psi_V) \lesssim 1$.
    Now, since $f_r \to f$ uniformly as $r \downarrow 0$, it follows from the weak lower semicontinuity that $f \in \mathcal{F}$ and that $\mathcal{E}(f) \lesssim \Xi(f)$. This proves the converse inclusion.

    Lastly, we check the comparability of the norms. Note that we have already verified the comparability of the energies $\mathcal{E}$ and $\Xi$. We therefore only need to compute estimates for the $L^Q$-norms. Let us note that the comparability of the $L^Q$-norms themselves is false.
    Let $f \in \mathcal{F} \cap C(X) = \mathcal{W} \cap C(X)$. By taking $B_d(x,r) = X$ in \eqref{eq:TwoMeasure}, we have the estimate
    \begin{align*}
        \left(\int_X \abs{f}^Q \, \text{d}\nu\right)^{\frac{1}{Q}} & \leq \left(\int_X \abs{f - f_{X,\mu}}^Q \, \text{d}\nu\right)^{\frac{1}{Q}} + \nu(X)^{\frac{1}{Q}} \left| \kint_X f\, \text{d}\mu \right|\\
        & \lesssim \mathcal{E}(f)^{\frac{1}{Q}} + \left(\int_{X} \abs{f}^Q \, \text{d}\mu\right)^{\frac{1}{Q}}.
    \end{align*}
    This shows the first inequality.
    For the second, we use similar computations
    \begin{align*}
        \left(\int_X \abs{f}^Q \, \text{d}\mu\right)^{\frac{1}{Q}} \lesssim \left(\int_X \abs{f - f_{X,\nu}}^Q \, \text{d}\mu\right)^{\frac{1}{Q}} +  \kint_X \abs{f}^Q\, \text{d}\nu
    \end{align*}
    and 
    \begin{align*}
        \left(\int_X \abs{f - f_{X,\nu}}^Q \, \text{d}\mu\right)^{\frac{1}{Q}} & \lesssim \left(\int_X \abs{f - f_{X,\mu}}^Q \, \text{d}\mu\right)^{\frac{1}{Q}} + \left(\int_X \abs{f - f_{X,\mu}}^Q \, \text{d}\nu\right)^{\frac{1}{Q}} \lesssim \Xi(f).
    \end{align*}
    In the second display we used $\PI(\beta)$, \eqref{eq:TwoMeasure}, and the comparability of $\mathcal{E}$ and $\Xi$.
    This concludes the proof.
\end{proof}

We also show a localized comparability of the $\mathcal{E}$ and $\Xi$.

\begin{proposition}\label{prop:Conf-Inv-local}
    Suppose that Assumption \ref{Assumption:Confdim} holds and assume $d_{\textup{H}}=\beta$.
    There is a constant $C \geq 1$ such that the following holds.
    For every $f \in \mathcal{F} \cap C(X)$ and any non-empty open set $U \subseteq X$ we have
    \begin{align*}
        C^{-1}\Gamma\Span{f}(U) \leq \sup_{O \Subset U} \limsup_{r \downarrow 0} \int_O \kint_{B_\rho(x,r)} \frac{\abs{f(x) - f(y)}^Q}{\nu(B_\rho(x,r))}  \, \textup{d}\nu(y) \, \textup{d}\nu(x) \leq C\Gamma\Span{f}(U).
    \end{align*}
    The supremum in the middle is taken over all non-empty open subsets $O$ whose closure is contained in $U$.
\end{proposition}

\begin{proof}
    Let $f$ and $U$ be as in the claim, and also let $\lambda$, $\mathcal{N}_r$, $\mathcal{V}_r$ and $\{ \psi_V \}_{V \in \mathcal{V}_r}$ be as in the proof of Proposition \ref{prop:Conf-Inv}.

    The second inequality is considered first. Let $O \Subset U$ and take $r > 0$ small enough so that, for any $V \in \mathcal{V}_r$, $V \cap O \neq \emptyset$ implies $\lambda V \subseteq U$. By a similar covering argument as in the proof of Proposition \ref{prop:Conf-Inv}, we have
    \begin{align*}
        & \quad \int_O \kint_{B_\rho(x,r)} \frac{\abs{f(x)-f(y)}^Q}{\nu(B_\rho(x,r))} \, \text{d}\nu(y)\, \text{d}\nu(x)\\
        \leq & \quad \sum_{\substack{ V \in \mathcal{V}_r \\ V \cap O \neq \emptyset } } \int_{V} \kint_{B_\rho(x,r)} \frac{\abs{f(x)-f(y)}^Q}{\nu(B_\rho(x,r))} \, \text{d}\nu(y)\, \text{d}\nu(x) \lesssim \Gamma\Span{f}(U).
    \end{align*}
    The claimed inequality follows by sending $r \downarrow 0$.

    Then we prove the remaining inequality.
    By the Borel regularity of $\Gamma\Span{f}$, there is an open subset $O \Subset U$ such that $\Gamma\Span{f}(U) \leq 2\Gamma\Span{f}(O)$. We also take another open set $\widetilde{O}$ so that $O \Subset \widetilde{O} \Subset U$.
    Now, for every small enough $r > 0$, $V \cap O \neq \emptyset$ implies $8 V \subseteq \widetilde{O}$, for every $V \in \mathcal{V}_\varepsilon$. Then, if $f_r$ is as given in the proof of Proposition \ref{prop:Conf-Inv}, its energy can be estimated
    \[
        \Gamma\Span{f_r}(U) \lesssim \int_{\widetilde{O}} \kint_{B_\rho(x,16 r)} \frac{\abs{f(x) - f(y)}^Q}{\nu(B_\rho(x,\varepsilon))}  \, \textup{d}\nu(y) \textup{d}\nu(x).
    \]
    This follows from a similar argument as in the proof of Proposition \ref{prop:Conf-Inv}.
    The remaining inequality in the claim now follows from the weak lower semicontinuity by sending $r \downarrow 0$.
    
\end{proof}

\subsection{On the Attainment}
This section concludes with a consequence of the attainment of conformal dimension.
We prove that the Hausdorff measure of an attaining metric is a minimal energy dominant measure.
An analogous result on the Sierpi\'nski carpet was obtained in \cite[Theorem 1.7]{MuruganShimizu}, and we use a similar method. 

\begin{corollary}\label{cor:nu-minimal}
    Suppose that Assumption \ref{Assumption:Confdim} and $d_{\textup{H}} = \beta$ hold. Let $\rho$ be a metric on $X$, quasisymmetric to $d$, so that $(X,\rho)$ is $Q$-Ahlfors regular.
    Then $(X,\rho)$ is $Q$-Loewner. Moreover, if $\nu$ is the $Q$-Hausdorff measure of $(X,\rho)$, then it is a minimal energy dominant measure of $(X,d,\mu,\mathcal{E},\mathcal{F},\Gamma)$.
\end{corollary}

\begin{proof}
    Note that, in the hypotheses, we assume that $(X,d)$ attains its conformal dimension and that $\rho$ is an attaining metric.
    During this proof, we consider the auxiliary $Q$-Sobolev space $(\rho,\nu,\Xi,\mathcal{W})$ where $\nu$ is the $Q$-Hausdorff measure of $(X,\rho)$.
    
    As noted in the proof of Corollary \ref{cor:confdim}, $(X,d)$ is combinatorially Loewner.
    Since $\rho$ is an attaining metric, $(X,\rho)$ is $Q$-Loewner. See \cite[Section 1.6]{CEB}.

    It remains to prove that $\nu$ is minimal energy dominant.
    Let $\mathcal{N}_r$ and $\mathcal{V}_r$ be as in the proof of Proposition \ref{prop:Conf-Inv}.
    By the $Q$-Ahlfors regularity and Proposition \ref{prop:Conf-Inv}, it holds for every $f \in \mathcal{F} \cap C(X)$ that
    \[
        \mathcal{E}(f) \sim \sup_{r > 0} \int_X \kint_{B_\rho(x,r)} \frac{\abs{f(x) - f(y)}^Q}{r^Q} \, \text{d}\nu(y)\text{d}\nu(x).
    \]
    In particular, $\Lip(X,\rho) \subseteq \mathcal{F}$.
    By Proposition \ref{prop:Conf-Inv-local}, we have
    \begin{equation}\label{eq:Lip-rho}
        \Gamma\Span{f} \ll \nu \text{ for all } f \in \Lip(X,\rho).
    \end{equation}

    Equation \eqref{eq:Lip-rho} is extended to general $f \in \mathcal{F}$ by an approximation argument. For $x \in X$ and $r > 0$ the function
    \[
        \phi(y) := \max\left\{ \min\left\{ \frac{2r - \rho(x,y)}{r},1 \right\},0\right\} \in \Lip(X,\rho)
    \]
    is Lipschitz function satisfying $\phi = 1$ on $B(x,r)$ and $\phi = 0$ in $X \setminus B(x,2r)$.
    This and \cite[Proof of Lemma 4.2]{ResistanceConjecture} shows that we can construct partitions of unity on $\{\psi_V \}_{V \in \mathcal{V}_r} \subseteq \Lip(X,\rho) \subseteq \mathcal{F}$ that have the same properties as in the proof of Proposition \ref{prop:Conf-Inv}. Note that $\mathcal{E}(\psi_V) \lesssim 1$ follows from Proposition \ref{prop:Conf-Inv-local}.
    
    For $r > 0$ and $f \in \mathcal{F}$ we define the approximations
    \[
        f_r := \sum_{V \in \mathcal{V}_r} \left( \kint_{V} f\, \text{d}\mu \right) \psi_V.
    \]
    It follows from similar arguments as in the proof of Proposition \ref{prop:Conf-Inv} that $\{f_r\}_{r > 0}$ is a bounded subset of $\mathcal{F}$. Moreover, $f_r \to f$ in $L^p(X,\mu)$ as $r \downarrow 0$. This follows by first proving the convergence for continuous functions via the uniform continuity, and then using the density $C(X) \subseteq L^p(X,\mu)$ and uniform boundedness of the linear operators $h \mapsto h_r$ in $L^p(X,\mu)$ for $r > 0$.
    Since each $f_r$ is Lipschitz, it follows from \eqref{eq:Lip-rho} and the weak lower semicontinuity that $\Gamma\Span{f} \ll \nu$ for a general $f \in \mathcal{F}$.

    The minimality of $\nu$ is argued next.
    Note that $(X,\rho)$ is bi-Lipschitz equivalent to a geodesic metric spaces because it is a Loewner space. See \cite[Theorem 8.3.2]{HKST} for details.
    Hence, without the loss of generality, we may assume that $(X,\rho)$ is geodesic.
    Now, choose a base point $x_0$ and let $f(x) := \rho(x,x_0)$.
    Take any point $x \in X \setminus \{x_0\}$ and let $r = \rho(x_0,x)$.
    Also take a geodesic curve $\gamma$ from $x_0$ to $x$, and let $\varepsilon \in (0,r/2)$. By the additivity of the length of curves, there is a point $z$ on $\gamma$ so that $\rho(z,x) = \varepsilon/4$ and $\rho(x_0,z) = \rho(x_0,x) - \varepsilon/4$. It then follows from the reverse triangle inequality and the $Q$-Ahlfors regularity that
    \begin{align*}
        \kint_{B_\rho(x,\varepsilon)} \frac{\abs{f(x) - f(y)}^Q}{\nu(B_\rho(x,\varepsilon))} \,\text{d}\nu(y)  \geq \frac{1}{\nu(B_\rho(x,\varepsilon))} \int_{B_\rho(z,\varepsilon/8)} \frac{(\varepsilon/16)^Q}{\nu(B_\rho(x,\varepsilon))} \, \text{d}\nu \geq C.
    \end{align*}
    We combine this with Proposition \ref{prop:Conf-Inv-local} to obtain $\Gamma\Span{f}(U) \geq D^{-1}\nu(U)$ for every open set $U \subseteq X$. By the Borel regularity, we get $\nu \ll \Gamma\Span{f}$, which clearly implies the minimality of $\nu$.
\end{proof}

\section{Singularity of energy measures}\label{sec:Singular}

This section regards the following singularity result of energy measures.

\begin{proposition}\label{prop:Singular}
    Suppose that Assumption \ref{Assumption:Confdim} holds, and let $\Lambda$ be any minimal energy dominant measure of $(X,d,\mu,\mathcal{E},\mathcal{F},\Gamma)$. If additionally $(X,d)$ is geodesic and $\beta > p$, then $\Lambda \perp \mu$.
\end{proposition}

The singularity of energy measures has its origins in the theory of diffusion on fractals.
The first result to this direction was obtained by Kusuoka on the Sierpi\'nski gasket \cite{KusuokaSingular}, and later Hino generalized it to more general fractals \cite{HinoSingular}.
In the full generality of Dirichlet forms on metric spaces, the singularity result was settled by Kajino and Murugan \cite{KajinoMuruganSingular}. By using a similar argument, Yang extended the singularity to $p$-energies \cite{yangsing}.
In fact, \cite{yangsing} essentially covers Proposition \ref{prop:Singular}, but we decided to provide a detailed proof since therein assumptions slightly differ from ours. For instance, we do not assume a Clarkson's inequality.

For the convenience of the reader not familiar with the singularity, we discuss a very simple special case of Proposition \ref{prop:Singular} which seems to provide quite accurate intuition for the general case. Suppose that $\Gamma\Span{f} \sim C\mu$ for $f \in \mathcal{F}$ and $C > 0$.
By a two-point estimate argument, see \cite[Theorem 3.2]{SobolevMetPoincare}, one can apply $\PI(\beta)$ to show that $f$ is $\alpha$-Hölder continuous for $\alpha = \beta/p > 1$.
Then, by composing $f$ with geodesic curves, one finds by differentiating that $f$ must be constant. Hence $\Gamma\Span{f} = 0$ by the strong locality, which is a contradiction to $C>0$.

Then in the general case, we do a blow-up argument at a density point of $\Gamma\Span{f}$ with respect to $\mu$.
By the Lebesgue differentiation theorem, $\Gamma\Span{f}$ behaves like $\mu$ near such points. By taking care of the error term and the potential singular part, we can perform a conceptually quite similar proof scheme as in the previous special case.

Lastly, let us note that Proposition \ref{prop:Singular} is false in general if $(X,d)$ is not geodesic. For instance, consider the $p$-Dirichlet space $(\R^n,d,\text{d}x,\mathcal{E},W^{1,p}(\R^n),\Gamma)$ so that $\mathcal{E}$ is the usual Dirichlet $p$-energy and $\text{d}\Gamma\Span{f} = \abs{\nabla f}^p \, \text{d}x$. The metric $d$ is chosen to be a snowflaked Euclidean metric.
The conclusion is of the proposition is obviously false, and it is easy to check that the hypotheses of Proposition \ref{prop:Singular} are all satisfied except that $(\R^n,d)$ is not geodesic.

\subsection{Infinitesimal harmonicity}
Our argument in the proof of Proposition \ref{prop:Singular} is essentially the same as those in \cite{KajinoMuruganSingular,yangsing}, but we streamline some of the details by adopting a method of Cheeger about infinitesimal harmonicty of Sobolev functions \cite[Theorem 3.7]{Cheeger}.
Let us note that, for the following lemma, we essentially only need the Vitali covering property of $\mu$, see \cite[Theorem 3.4.3]{HKST}, and the weak lower semicontinuity.
An analogous fact was pointed out in the beginning of Cheeger's paper.

\begin{lemma}\label{lemma:Cheeger}
    Suppose that Assumption \ref{Assumption:Confdim} holds, and let $f \in \mathcal{F} \cap C(X)$. It holds for $\mu$-almost every $x \in X$ that 
    \[
        \lim_{r \downarrow 0} \left( \frac{1}{\mu(\overline{B}(x,r))} \int_{\overline{B}(x,r)} \, \textup{d}\Gamma\Span{f} - \inf_{ \varphi } \frac{1}{\mu(\overline{B}(x,r))}\int_{\overline{B}(x,r)} \, \textup{d}\Gamma\Span{f + \varphi}  \right) = 0,
    \]
    where the infimum is taken over all $\varphi \in \mathcal{F}$ with $\varphi = 0$ in $X \setminus B(x,r)$.
\end{lemma}

\begin{proof}
    We make a counter-assumption that the conclusion is false. In that case, there is $\varepsilon>0$ and a Borel subset $K \subseteq X$ with $\mu(K) > 0$ so that the following holds. For every $z \in K$ there are sequences $\{ r_{z,i} \}_{i = 1}^\infty \subseteq (0,1)$ and $\{\varphi_{z,i} \}_{i = 1}^\infty \subseteq \mathcal{F}$ with $r_{z,i} \downarrow 0$, $\varphi_{i,z} = 0$ in $X\setminus B(z,r_{z,i})$ and
    \begin{equation}\label{eq:Asymp-Harmonic-varepsilon}
        \frac{1}{\mu(\overline{B}(z,r_{z,i}))} \int_{\overline{B}(z,r_{z,i})} \, \textup{d}\Gamma\Span{f} \geq \varepsilon +  \frac{1}{\mu(\overline{B}(z,r_{z,i}))}\int_{\overline{B}(z,r_{z,i})} \, \textup{d}\Gamma\Span{f + \varphi_{z,i}}.
    \end{equation}
    By the Vitali covering property, for every $\delta > 0$ there is a countable subset $J \subseteq K$ such that the following holds. There is a collection of closed balls $\{ \overline{B}(z,r_z) \}_{z \in J}$ and a collection of functions $\{ \varphi_{z} \}_{z \in J} \subseteq \mathcal{F}$ such that $r_z = r_{z,i}$ and $\varphi_z = \varphi_{z,i}$ for some $i$ which may depend on $z$, and $r_z < \delta$ for all $z \in J$. In addition, they are chosen to cover almost all of $K$. But then we can choose a finite subcollection that covers at least half of $K$. In other words, there is a finite subset $J_0 \subseteq J$ with
    \[
      2\mu\left(\bigcup_{z \in J_0}\overline{B}(z,r_z)\right) \geq \mu(K).  
    \]
    
    Now, fix $\delta > 0$ and the data as above. By modifying $\varphi_z$ if necessary, we may assume that $m_z \leq f + \varphi_z \leq M_z$ in $B(z,r_z)$ where $m_z,M_z$ are the minimum and maximum of $f$ in $B(z,2r_z)$, respectively. For instance, we can use the continuity of $f$ to construct $\tilde{\varphi}_z \in \mathcal{F}$ so that
    \[
        \tilde{\varphi}_z = 
        \begin{cases}
            - f + \max \{ \min\{ f + \varphi_z,m_z \} ,M_z \} & \text{ in } B(z,r_z) \\
            0 & \text{ elsewhere}.
        \end{cases}
    \]
    Note that the estimate \eqref{eq:Asymp-Harmonic-varepsilon} holds if we replace $\varphi_z$ with $\tilde{\varphi}_z$ by the Lipschitz contractivity of the $p$-Dirichlet space. Then consider the function
    \[
        f_\delta := f + \sum_{z \in J_0} \varphi_z \in \mathcal{F}.
    \]
    By the strong locality and \eqref{eq:Asymp-Harmonic-varepsilon}, its energy can be estimated 
    \begin{align*}
        \mathcal{E}(f_\delta) & = \int_{X \setminus \bigcup_{z \in J_0} \overline{B}(z,r_z) } \, \text{d}\Gamma\Span{f} + \sum_{z \in J_0} \int_{\overline{B}(z,r_z)} \text{d}\Gamma\Span{f + \varphi_z}\\
        & \leq \int_{X \setminus \bigcup_{z \in J_0} \overline{B}(z,r_z) } \, \text{d}\Gamma\Span{f} + \sum_{z \in J_0} \int_{\overline{B}(z,r_z)} \text{d}\Gamma\Span{f} - \varepsilon\sum_{z \in J_0} \mu(\overline{B}(z,r_z))\\
        &  \leq \mathcal{E}(f) - \frac{\varepsilon \mu(K)}{2}.
        \end{align*}
    Since $f$ is continuous and $X$ is assumed to be compact, $f_\delta \to f$ uniformly. For this to hold, the modifications of $\varphi_z$ are necessary.
    By sending $\delta \downarrow 0$, the weak lower semicontinuity yields a contradiction $\mathcal{E}(f) \leq \mathcal{E}(f) - \varepsilon \mu(K)/2$ .
\end{proof}

\subsection{Proof of the singularity}
Our proof uses a cutoff Sobolev inequality. See \cite[Introduction]{ResistanceConjecture} for further background.

\begin{lemma}\label{lemma:CS}
    Assume the hypotheses of Lemma \ref{prop:Singular}, and let $\delta \in (0,1)$.
    There is $C \geq 1$, depending on $\delta$, such that for every $x \in X$ and $r > 0$ there exists $\varphi \in \mathcal{F}$ satisfying $0 \leq \varphi \leq 1$, $\varphi = 1$ in $B(x,(1-\delta) r)$, $\varphi = 0$ in $X \setminus B(x,r)$ and the following cutoff Sobolev inequality $\CS(\Psi)$. If $f \in \mathcal{F} \cap C(X)$ then
    \[
        \int_{\overline{B}(x,(1+\delta)r)} \abs{f}^p \, \textup{d}\Gamma\Span{\varphi} \leq C\left( \int_{\overline{B}(x,(1+\delta)r) \setminus B(x,(1-\delta)r)} \, \textup{d}\Gamma\Span{f} + \frac{1}{r^\beta} \int_{B(x,2r)} \abs{f}^p \, \textup{d}\mu \right).
    \]
\end{lemma}

\begin{proof}
    By \cite[Lemma 2.13 and Theorem 3.3]{ResistanceConjecture} we have the following. There are $C,\sigma \geq 1$ such that for every $B = B(x,r)$ there exists $\varphi_{B} \in \mathcal{F}$ satisfying $0 \leq \varphi_B \leq 1$, $\varphi_B = 1$ in $B$, $\varphi_B = 0$ in $X \setminus 2B$ and
    \[
        \int_{B} \abs{f}^p \, \text{d}\Gamma\Span{\varphi_B} \leq C\left( \int_{\sigma B} \text{d}\Gamma\Span{f} + \frac{1}{r^\beta} \int_{2B} \abs{f}^p \, \text{d}\mu \right).
    \]
    Now, fix $x \in X$ and $r > 0$.
    Let $\lambda \geq 4$ be a large enough constant depending on $\delta,\sigma$, and take an $(r/\lambda)$-net $\mathcal{N}$ of $B(x,r)$.
    For each $z \in \mathcal{N}$ let $\varphi_z := \varphi_{B(z,r/\lambda)}$ be as given above, and let $\varphi := \max_{z \in \mathcal{N}} \varphi_z \in \mathcal{F}$.
    By choosing $\lambda$ to be large enough, we have the implication $\varphi = 1$ on $B(z,r/\lambda)$ unless $B(z,\sigma r/\lambda) \subseteq \overline{B}(x,(1+\delta)r) \setminus B(x,(1-\delta)r)$.
    The claimed cutoff Sobolev inequality now follows from the strong locality by applying $\Gamma\Span{\varphi} \leq \sum_{z \in \mathcal{N}} \Gamma\Span{\varphi_z}$.

\end{proof}

\begin{lemma}\label{lemma:Singular}
    Assume the hypotheses of Proposition \ref{prop:Singular}, and let $f \in \mathcal{F} \cap C(X)$. Then $\Gamma\Span{f} \perp \mu$. 
\end{lemma}

\begin{proof}
    Suppose that the conclusion of the lemma is false. In other words, in the Lebesgue decomposition $\Gamma\Span{f} = \mu_a + \mu_s$ where $\mu_a \ll \mu$ and $\mu_s \perp \mu$, it holds that $\mu_a \neq 0$.

    The main method of the proof is a (quantitative) blow-up argument, and the first step is to choose the base point.
    Let $x \in X$ so that the conclusion of Lemma \ref{lemma:Cheeger} and the following hold. If $h$ is the Radon--Nikodym derivative of $\mu_a$ w.r.t. $\mu$ then
        \begin{equation}\label{eq:Singular-LDT}
            \lim_{r \downarrow 0}
            \kint_{B(x,r)} \abs{h - h(x)} \, \text{d}\mu
            = 0 \text{ and }
            \lim_{r \downarrow 0} 
            \frac{\mu_s(B(x,r))}{\mu(B(x,r))} = 0.
        \end{equation}
    This is a standard property of doubling measures on metric spaces.
    After a rescaling, we assume without loss of generality that $h(x) = 1$.

    Since $(X,d)$ is geodesic and $\mu$ is doubling, it follows from the annular decay estimate \cite[Proposition 11.5.3]{HKST} that, for all $r > 0$ and $\delta > 0$, we have
    \begin{equation}\label{eq:Annular-decay}
        \mu( \overline{B}(x,(1+\delta)r) \setminus B(x,(1-\delta)r) ) \leq C_1\delta^{\alpha} \mu(B(x,r)).
    \end{equation}
    For the remainder of the the proof, we fix $\delta \in (0,1/8)$ so that the coefficient in the right-hand side is suitably small. We will need to take care with constants, and we denote by $C_i$ constants that only depend on the doubling constant of the measure and $p$.
    
    The second step is to choose a suitably scale in the blow-up procedure.
    Fix a large constant $\lambda \geq 8\delta^{-1}$ and a suitably small $\varepsilon > 0$ depending on $\lambda$. The precise values of $\delta,\lambda,\varepsilon$ are determined at the end of the proof.
    By \eqref{eq:Singular-LDT} and Ahlfors regularity \eqref{eq:AR}, for all small enough $r > 0$ and all $y \in B(x,2r)$, we have
    \begin{equation}\label{eq:blow-up-Singular}
        \left|\frac{\Gamma\Span{f}(B(y,3\sigma r/\lambda))}{\mu(B(y,3\sigma r/\lambda))}-1\right| \leq C_2 \lambda^{d_{\textrm{H}}} \varepsilon.
    \end{equation}

    The third step is to find a suitably competitor as in Lemma \ref{lemma:Cheeger}.
    Write $\overline{B} := \overline{B}(x,(1+\delta)r)$ and $A:= \overline{B}(x,(1+\delta)r) \setminus B(x,(1-\delta)r)$.
    By the cutoff Sobolev inequality in Lemma \ref{lemma:CS} there is $\varphi \in \mathcal{F} \cap C(X)$ such that $0 \leq \varphi \leq 1$, $\varphi = 1$ in $B(x,(1-\delta)r)$, $\varphi = 0$ in $X \setminus B(x,r)$ and
    \begin{equation}\label{eq:CS-Singular}
        \int_{\overline{B}} \abs{h}^p \, \text{d}\Gamma\Span{\varphi} \leq C_3\left(\int_{A} \text{d}\Gamma\Span{h} + \frac{1}{r^\beta} \int_{B(x,2r)} \abs{h}^p \, \text{d}\mu \right)
    \end{equation}
    for all $h \in \mathcal{F} \cap C(X)$.
    
    We now begin the actual computations. First estimate
    \begin{align*}
        & \quad \Gamma\Span{ f + (f_{B(x,3r/\lambda)} - f)\varphi }(\overline{B}) = \Gamma\Span{ f + (f_{B(x,3r/\lambda)} - f) \varphi}(A) \\
        \leq & \quad 2^{p-1}\left( \Gamma\Span{f}(A ) + \Gamma\Span{(f_{B(x,3r/\lambda)} - f)\varphi}(A)\right)\\
        \leq & \quad 4^{p-1} \left( \Gamma\Span{f}(A) + \int_{A} \abs{f - f_{B(x,3r/\lambda)}}^p \, \text{d}\Gamma\Span{\varphi} + \int_{A} \abs{\varphi}^p \, \text{d}\Gamma\Span{f} \right)\\
        \leq & \quad 4^p \left( \Gamma\Span{f}(A) + \int_{\overline{B}} \abs{f - f_{B(x,3r/\lambda)}}^p \, \text{d}\Gamma\Span{\varphi} \right)\\
        \leq & \quad 4^p C_3 \left( \Gamma\Span{f}(A) + \frac{1}{r^\beta} \int_{B(x,2r)} \abs{f - f_{B(x,3r/\lambda)}}^p \, \text{d}\mu \right).
    \end{align*}
    The first row follows from the strong locality, and the second from the triangle inequality and $\abs{a + b}^p \leq 2^{p-1}(\abs{a}^p + \abs{b}^p)$.
    We used the Leibniz rule in the third step, cf. \cite[Lemma 2.18]{EID2025},  $\abs{\varphi} \leq 1$ in the fourth step, and \eqref{eq:CS-Singular} in the fifth step.

    The second term on the right hand side is estimated as follows. Let $\mathcal{N} \subseteq B(x,2r)$ be an $(r/\lambda)$-net. Since $(X,d)$ is geodesic, we find for each $z \in \mathcal{N}$ a sequence $\{y_i\}_{i = 0}^k$ such that $y_0 = x$, $y_k = z$, $d(y_i,y_{i+1}) \leq r/\lambda$ for all $i$ and $k \leq 2\lambda$.
    Using this, we have
    \begin{align*}
        & \quad \, \int_{B(z,r/\lambda)} \abs{f - f_{B(x,3r/\lambda)}}^p \, \text{d}\mu\\
        \leq & \quad 2^{p-1}\left( \int_{B(z,r/\lambda)}\abs{f-f_{B(z,3r/\lambda)}}^p + \left( \sum_{i = 1}^k\abs{f_{B(y_{i-1},3r/\lambda)} - f_{B(y_{i},3r/\lambda)}} \,  \right)^p \text{d}\mu \right)\\
        \leq & \quad C_4 \mu(B(z,r/\lambda)) \left(\frac{r}{\lambda}\right)^\beta \left( \frac{\Gamma\Span{f}(B(z,3\sigma r/\lambda))}{\mu(B(z,3\sigma r/\lambda))} +   \left( \sum_{i = 1}^k \left\{\frac{\Gamma\Span{f}(B(y_i,3\sigma r/\lambda))}{\mu(B(y_i,3\sigma r/\lambda))}\right\}^{\frac{1}{p}}   \right)^p \right)\\
        \leq & \quad C_5 (\lambda^{d_{\textrm{H}}}\varepsilon + 1)\mu(B(z,r/\lambda)) k^p \left(\frac{r}{\lambda}\right)^\beta \leq 2C_5(\lambda^{d_{\textrm{H}}}\varepsilon + 1) r^\beta \lambda^{p-\beta} \mu(B(z,r/\lambda)).
    \end{align*}
    We used $\PI(\beta)$ in the third row and \eqref{eq:blow-up-Singular} in the fourth.
    By summing over all $z \in \mathcal{N}$
    \begin{align*}
         & \quad \,\frac{1}{r^\beta} \int_{B(x,r)} \abs{f - f_{B(x,3r/\lambda)}}^p \, \text{d}\mu\leq \frac{1}{r^\beta} \sum_{z \in \mathcal{N}} \int_{B(z,r/\lambda)} \abs{f - f_{B(x,3r/\lambda)}}^p \, \text{d}\mu\\
         \leq & \quad 2C_5(\lambda^{d_{\textrm{H}}}\varepsilon + 1) \lambda^{p-\beta}\sum_{z \in \mathcal{N}} \mu(B(z,r/\lambda)) \leq C_6 (\lambda^{d_{\textrm{H}}}\varepsilon + 1) \lambda^{p-\beta} \mu(\overline{B}).
    \end{align*}
    We used bounded overlapping in the the last inequality.
    
    Then estimate $\Gamma\Span{f}$ using \eqref{eq:Singular-LDT} and \eqref{eq:Annular-decay}.
    By taking $\mathcal{N}$ to be a $(3\sigma r/\lambda)$-net of $A$, we have
    \begin{align*}
        \Gamma\Span{f}(A) & \leq \sum_{z \in \mathcal{N}} \Gamma\Span{f}(B(z,3\sigma r/\lambda))\\
        & \leq C_2(\lambda^{d_{\textrm{H}}}\varepsilon + 1) \sum_{z \in \mathcal{N}} \mu(B(z,3\sigma r/\lambda)) \\
        & \leq C_7 (\lambda^{d_{\textrm{H}}}\varepsilon + 1) \mu(\overline{B}(x,(1+2\delta)r) \setminus B(x,(1- \delta/2) r ))\\
        & \leq C_8 (\lambda^{d_{\textrm{H}}}\varepsilon + 1) \delta^\alpha \mu(\overline{B}).
    \end{align*}

    By combining the estimates derived during the proof, we have
    \[
        \Gamma\Span{ f + (f_{B(x,3r/\lambda)} - f)\varphi }(\overline{B}) \leq C_9(\lambda^{d_\textrm{H}}\varepsilon + 1)\delta^\alpha \mu(\overline{B}) + C_{10} (\lambda^{d_{\textrm{H}}}\varepsilon + 1)\lambda^{p-\beta}\mu(\overline{B}).
    \]
    We now fix the choices of our constants.
    First, $\delta \in (0,1/8)$ is chosen so that $2C_9 \delta^\alpha < 1/4$. Then choose $\lambda \geq 8\delta^{-1}$ so that $2C_{10} \lambda^{p-\beta} < 1/4$. 
    Note that here it is essential that $\beta > p$.
    Lastly, choose $\varepsilon > 0$ so that $\lambda^{d_{\textrm{H}}}\varepsilon < 1$.
    On the other hand, by \eqref{eq:Singular-LDT} and the fact that $h(x) = 1$, we can choose $r > 0$ so that $\mu(\overline{B}) \leq 2\Gamma\Span{f}(\overline{B}) - 1/4\mu(\overline{B})$.
    Then we have
    \[
        \Gamma\Span{ f + (f_{B(x,3r/\lambda)} - f)\varphi }(\overline{B}) \leq \frac{\mu(\overline{B})}{2} \leq \Gamma\Span{f}(\overline{B}) - \frac{\mu(\overline{B})}{8}. 
    \]
    Since these choice are valid for arbitrarily small $r > 0$, the previous inequality contradicts Lemma \ref{lemma:Cheeger}, which completes the proof.
\end{proof}

\begin{proof}[Proof of Proposition \ref{prop:Singular}]
    Since every pair of minimal energy dominant measures are mutually absolutely continuous by definition, it is sufficient to prove the claim for any choice of $\Lambda$.
    Hence, by the proof of Proposition \ref{prop:MinEnergyDom}, it is sufficient to prove that $\Gamma\Span{f} \perp \mu$ for every $f \in \mathcal{F}$. Note that Lemma \ref{lemma:Singular} gives this for continuous $f$.
    
    Let $f \in \mathcal{F}$, and recall the partitions of unity $\{ \psi_U \}_{U \in \mathcal{U}_\varepsilon}$ from Assumption \ref{Assumption:Confdim}.
    Consider the approximations
    \[
        f_\varepsilon := \sum_{U \in \mathcal{U}_\varepsilon} \left( \kint_{U} f\, \text{d}\mu \right)\psi_U.
    \]
    By a similar argument as in the proof of Corollary \ref{cor:nu-minimal}, it holds that $f_r \to f$ in $L^p(X,\mu)$, and that $\mathcal{E}(f_r) \lesssim \mathcal{E}(f)$. Since $\Gamma\Span{f_r} \perp \mu$ for all $r > 0$, the mutual singularity $\Gamma\Span{f} \perp \mu$ now follows from the weak lower semicontinuity.
\end{proof}

\section{Non-attainment of conformal dimension}\label{sec:NonAttain}
This section proves the main result, Theorem \ref{main:Carpet}, about the non-attainment of conformal dimension.
It is generalized into the following, and see Remark \ref{remark:Explain} for a detailed explanation on how it implies Theorem \ref{main:Carpet}.
Uniform scalability is given in Definition \ref{def:SSS}.
\begin{theorem}\label{thm:non-attainment}
    Let $Q > 1$, $\beta \geq Q$, and $(X,d_X,\mu_X,\mathcal{E}_X,\mathcal{F}_X,\Gamma_X)$ be a $(Q,\beta)$-Poincar\'e--Dirichlet space satisfying Assumption \ref{Assumption:Confdim}.
    Also let $k \in \N$ with $k \geq 2$, and suppose $k \cdot d_{\textup{H}} = \beta$ where $d_{\textup{H}}$ is the Hausdorff dimension of $(X,d_X)$. Finally let $(Z,d)$ be $k$-fold Cartesian product of $(X,d_X)$ where $d$ is the $\ell^\infty$-metric.
    Then $Q$ is the conformal dimension of $(Z,d)$.
    Moreover, if $\beta > Q$ and $(X,d_X)$ is geodesic and uniformly scalable, then $(Z,d)$ does not attain its conformal dimension.
\end{theorem}

Theorem \ref{thm:non-attainment} is essentially a special case of Theorem \ref{thm:non-attainment-prod} where $Y = X^{k-1}$.

\begin{theorem}\label{thm:non-attainment-prod}
    Let $Q > 1$ and $\beta \geq Q$. Suppose that $(X,d_X,\mu_X,\mathcal{E}_X,\mathcal{F}_X,\Gamma_X)$ and $(Y,d_Y,\mu_Y,\mathcal{E}_Y,\mathcal{F}_Y,\Gamma_Y)$ are $(Q,\beta)$-Poincar\'e--Dirichlet spaces, both satisfying Assumption \ref{Assumption:Confdim}. Further, suppose that $\beta = d_{\textup{H},X} + d_{\textup{H},Y}$ where $d_{\textup{H},X}$ and $d_{\textup{H},Y}$ are Hausdorff dimensions of $(X,d_X)$ and $(Y,d_Y)$, respectively.
    Finally, let $(Z,d)$ be the Cartesian product of $(X,d_X)$ and $(Y,d_Y)$ where $d$ is the $\ell^\infty$-metric.
    Then $Q$ is the conformal dimension of $(Z,d)$.
    Moreover, if $\beta > Q$ and $(X,d_X)$, $(Y,d_Y)$ are geodesic and uniformly scalable, then $(Z,d)$ does not attain its conformal dimension.
\end{theorem}

Taking Theorem \ref{thm:non-attainment-prod} as given, we prove Theorem \ref{thm:non-attainment} as follows.

\begin{proof}[Proof of Theorem \ref{thm:non-attainment}] By repeatedly applying Theorem \ref{thm:Z}, we obtain a $(Q,\beta)$-Poincar\'e--Dirichlet space $(Y,d_Y,\mu_Y,\mathcal{E}_Y,\mathcal{F}_Y,\Gamma_Y)$ so that $(Y,d_Y)$ is the $(k-1)$-fold Cartesian product of $(X,d_X)$.
It is routine verify the conditions in Assumption \ref{Assumption:Confdim}, and that the Hausdorff dimension of $(Y,d_Y)$ is equal to $d_{\textrm{H},Y} := (k-1)\cdot d_{\textrm{H}}$.
Hence, if $(Z,d)$ is the $k$-fold Cartesian product, it follows from Theorem \ref{thm:non-attainment-prod} that $Q$ is the conformal dimension of $(Z,d)$.

Since $(X,d_X)$ is geodesic, so is $(Y,d_Y)$.
It is direct to check that $(Y,d_Y)$ is uniformly scalable if $(X,d_X)$ is.
Since $\beta > Q$ by the hypotheses, the non-attainment of the $k$-fold product $(Z,d)$ now follows from Theorem \ref{thm:non-attainment-prod}.
    
\end{proof}

\subsection{Blow-up procedure}
The remaining objective of the work is the proof of Theorem \ref{thm:non-attainment-prod}.
Most of the techniques required have already been developed in the previous sections. Our argument requires one additional ingredient, which is a suitable blow-up method.
In the following definition, we introduce a general self-similarity condition under which we can take nice limits along blow-up sequences.
This condition is essentially a variant of the notion of approximate self-similarity introduced in \cite{KleinerICM}.

\begin{definition}\label{def:SSS}
    We say that a metric space $(X,d)$ is \emph{uniformly scalable} if there is a constant $L \geq 1$ such that, for every $x \in X$ and $\varepsilon > 0$, there is $\Phi_{x,\varepsilon} : X \to B(x,\varepsilon)$ satisfying the following conditions.
    \begin{enumerate}
        \item $x \in \Phi_{x,\varepsilon}(X)$.
        \item For every $y \in X$ and $r \in (0,2]$ there is $z \in X$ such that 
        \[
        \Phi_{x,\varepsilon}(B(y,r)) \supseteq B(z,\varepsilon r/L).
        \]
        \item For every $y,z \in X$ we have
        \[
            L^{-1} \frac{d(y,z)}{\varepsilon} \leq d(\Phi_{x,\varepsilon}(y),\Phi_{x,\varepsilon}(z)) \leq L \frac{d(y,z)}{\varepsilon}.
        \]
    \end{enumerate}
\end{definition}

Assume the hypotheses of Theorem \ref{thm:non-attainment-prod}.
For brevity, we denote by $\Phi_{x,\varepsilon}$ the maps provided by the uniform scalability, regardless whether the ambient space is $X$ or $Y$.
For $x = (x_1,x_2) \in Z$ we define the map $\Phi_{x,\varepsilon} : Z \to B(z,\varepsilon)$ by
\[
    (y_1,y_2) \mapsto (\Phi_{x_1,\varepsilon}(y_1),\Phi_{x_2,\varepsilon}(y_2)).
\]

Recall that a Radon measure has \emph{full topological support} if every non-empty open set has positive measure.
If $\textrm{m}$ is a Radon measure on $(Z,d)$ of full topological support, we denote the normalized push-forward measure by
\[
    \mathrm{m}_{x,\varepsilon} := \frac{ (\Phi_{x,\varepsilon}^{-1})_*(\textrm{m}) }{ \textrm{m}(\Phi_{x,\varepsilon}(Z)) }.
\]

\begin{lemma}\label{lemma:Blow-up-1}
    Suppose that the hypotheses of Theorem \ref{thm:non-attainment-prod} hold.
    Let $\Lambda_X$ be a Radon measure on $(X,d_X)$ with full topological support, $\textup{m} := \Lambda_X \otimes \mu_Y$, $x = (x_1,x_2) \in Z$ and $\{\varepsilon_i\}_{i = 1}^\infty \subseteq (0,1)$ with $\varepsilon_i \downarrow 0$.
    Then there is a subsequence $\{\tilde{\varepsilon}_j\}_{j = 1}^\infty  \subseteq \{\varepsilon_i\}_{i = 1}^\infty$ and a Radon measure $\tilde{\Lambda}_X$ on $(X,d_X)$ such that $\textup{m}_{x,\tilde{\varepsilon}_j}$ converges weakly to a Radon measure which is mutually absolutely continuous with $\tilde{\Lambda}_X \otimes \mu_Y$.
\end{lemma}

\begin{proof}
    Note that $\textrm{m}_{x,\varepsilon}$ can be written as
    \[
        \textrm{m}_{x,\varepsilon} = \left( \frac{ (\Phi_{x_1,\varepsilon}^{-1})_*(\Lambda_X) }{ \Lambda_X(\Phi_{x_1,\varepsilon}(X)) } \right) \otimes \left( \frac{ (\Phi_{x_2,\varepsilon}^{-1})_*(\mu_Y) }{ \mu_Y(\Phi_{x_2,\varepsilon}(Y)) }  \right)
    \]
    By the compactness of $(X,d_X)$ and $(Y,d_Y)$, there is a subsequence $\{ \tilde{\varepsilon}_j \}_{j = 1}^\infty \subseteq \{\varepsilon_i\}_{i = 1}^\infty$ such that the marginals of $\textrm{m}_{x,\tilde{\varepsilon}_j}$ converge weakly to some probability Radon measures $\tilde{\Lambda}_X$ and $\tilde{\mu}_Y$, respectively.
    Then $\textrm{m}_{x,\tilde{\varepsilon}_j}$ converges weakly to $\tilde{\Lambda}_X \otimes \tilde{\mu}_Y$. Moreover, the second marginals in the above display are uniformly comparable to the $d_{\textrm{H},Y}$-Hausdorff measure $\mu_Y$ of $(Y,d_Y)$.
    This follows from Definition \ref{def:SSS}-(3) and the $d_{\textrm{H},Y}$-Ahlfors regularity of $(Y,d_Y)$. In particular, $\tilde{\mu}_Y$ is also comparable to $\mu_Y$, which shows the required mutual absolute continuity.
\end{proof}


\begin{lemma}\label{lemma:Blow-up-2}
    Suppose that the hypotheses of Theorem \ref{thm:non-attainment-prod} hold.
    Let $\rho$ be a metric on $Z$, quasisymmetric to $d$, so that $(Z,\rho)$ is $Q$-Ahlfors regular and $\{\varepsilon_i\}_{i = 1}^\infty \subseteq (0,1)$ with $\varepsilon_i \downarrow 0$. Let $\nu$ be the $Q$-Hausdorff measure of $(Z,\rho)$ and $x \in Z$.
    Then there is a metric $\tilde{\rho}$ on $Z$, quasisymmetric to $d$, a Radon probability measure $\tilde{\nu}$ on $(Z,d)$, and a subsequence $\{ \tilde{\varepsilon}_j \}_{j = 1}^\infty \subseteq \{\varepsilon_i\}_{i = 1}^\infty$ so that the following hold.
    The metric space $(Z,\tilde{\rho})$ is $Q$-Ahlfors regular, the Radon measures $\nu_{x,\tilde{\varepsilon}_j}$ converge weakly to $\tilde{\nu}$ and $\tilde{\nu}$ is comparable to the $Q$-Hausdorff measure of $(Z,\tilde{\rho})$.
\end{lemma}

\begin{proof}
    The subscript $x$ is dropped during the proof to ease the notation. We also denote $D_\varepsilon := \diam_{\rho}(\Phi_{\varepsilon}(Z))$.
    The constant $L$ is as in Definition \ref{def:SSS}, and $\rho$ is assumed to be $\eta_0$-quasisymmetric to $d$.
    Without loss of generality, we assume $\diam_d(Z) = 1$.
    
    The metric $\tilde{\rho}$ is constructed first.
    For $\varepsilon > 0$ consider the pull-back metrics
    \[
        \rho_{\varepsilon}(y,z) :=  D_\varepsilon^{-1}\rho(\Phi_{\varepsilon}(y),\Phi_{\varepsilon}(z)).
    \]
    Then the metric spaces $(Z,\rho_{\varepsilon})$ have diameters equal to 1.
    Moreover, a direct computation shows that $\rho_{\varepsilon}$ are $\eta$-quasisymmetric to $d$ where $\eta$ only depends on $\eta_0$ and $L$.
    Using this and $\diam_d(Z) = 1$, we compute
    \begin{align*}
        \abs{\rho_{\varepsilon}(y_1,z_1) - \rho_{\varepsilon}(y_2,z_2)} & \leq \rho_{\varepsilon}(y_1,y_2) + \rho_{\varepsilon}(z_1,z_2)\\
        & \leq \eta\left(2d(y_1,y_2) \right) + \eta\left(2d(z_1,z_2) \right).
    \end{align*}
    The previous display and the continuity of $\eta$ implies that $\{ \rho_{\varepsilon} \}_{\varepsilon \in (0,1)} \subseteq C(Z \times Z)$ is equicontinuous.
    We also have a lower-bound estimate
    \[
        \rho_{\varepsilon}(y,z) \geq \frac{1}{2} \tilde{\eta}\left( d(y,z) \right)
    \]
    where $\tilde{\eta} : [0,\infty) \to [0,\infty)$ is an increasing homeomorphism depending only on $\eta$. This follows from \cite[Proposition 10.6]{He} where it is shown that quasisymmetry is a quantitative symmetric relation.
    In particular, $\inf_{\varepsilon > 0} \rho_\varepsilon(x,y) > 0$ whenever $x \neq y$. Now, by applying the Arzela--Ascoli theorem, we obtain a metric $\tilde{\rho}$ on $Z$ and a subsequence $\{ \tilde{\varepsilon}_j\}_{j = 1}^\infty \subseteq \{ \varepsilon_i\}_{i = 1}^\infty$ so that $\rho_{\tilde{\varepsilon}_j}$ converges uniformly to $\tilde{\rho}$.
    It follows from the uniform convergence that $\tilde{\rho}$ is quasisymmetric to $d$.

    By taking a subsequence of $\{\tilde{\varepsilon}_j\}_{j = 1}^\infty$ if necessary, we may assume that $\nu_{x,\tilde{\varepsilon}_j}$ has a weak limit $\tilde{\nu}$.
    The remaining conclusions in the claim follow once we verify that $(X,\tilde{\rho},\tilde{\nu})$ is $Q$-Ahlfors regular in the sense that there is $C \geq 1$ such that
    \[
       C^{-1} r^Q \leq \nu( B_{\tilde{\rho}}(y,r) ) \leq C r^Q \text{ for every } x \in X \text{ and } r \in (0,2].
    \]
    We first prove the analogous property for $(X,\rho_{\varepsilon},\nu_{\varepsilon})$.
    Let $y,r$ be as in the display. First, we note $\Phi_\varepsilon(Z)$ contains a ball $B_\rho(z,\delta D_\varepsilon)$ where $\delta \in (0,1)$ depends only on $L$ and $\eta_0$. This follows from \eqref{eq:QS-balls} by applying (2) and (3) in Definition \ref{def:SSS}, and the fact that $\rho$ is $\eta_0$-quasisymmetric to $d$.
    The upper-bound is obtained from
    \[
        \nu_{\varepsilon}(B_{\rho_{\varepsilon}}(y,r)) = \frac{\nu( B_\rho(\Phi_{\varepsilon}(y),r D_\varepsilon )  ) \cap \Phi_{\varepsilon}(Z)) )}{\nu(\Phi_\varepsilon(Z))} \leq \frac{\nu( B_\rho(\Phi_{\varepsilon}(y),r D_\varepsilon ))}{\nu(B_\rho(z,\delta D_\varepsilon))} \leq Cr^Q.
    \]
    Again, by (2) and (3) in Definition \ref{def:SSS} and $\rho$ being $\eta_0$-quasisymmetric to $d$, we have
    \[
        \Phi_{\varepsilon}(B_{\rho_{\varepsilon}}(y,r)) \supseteq B_\rho(z,\delta r D_\varepsilon )
    \]
    for some $z \in X$, where $\delta \in (0,1)$ depends only on $L$ and $\eta_0$. The lower-bound now follows from
    \[
        \nu_{\varepsilon}(B_{\rho_{\varepsilon}}(y,r)) = \frac{\nu( B_\rho(\Phi_{\varepsilon}(y),r D_\varepsilon )  ) \cap \Phi_{\varepsilon}(Z)) )}{\nu(\Phi_\varepsilon(Z))} \geq \frac{\nu(B_\rho(z,\delta r D_\varepsilon ))}{\nu(B_\rho(z,2D_\varepsilon))} \geq C^{-1} r^Q.
    \]
    Finally, the desired $Q$-Ahlfors regularity follows by sending $j \to \infty$ and $\tilde{\varepsilon}_j \downarrow 0$ because $C$ in the above inequalities does not depend on $\varepsilon$.
\end{proof}

\subsection{Proof of the non-attainment}

We are finally ready to prove the non-attainment result for product spaces.

\begin{proof}[Proof of Theorem \ref{thm:non-attainment-prod}]
    Let $(Z,d,\mu,\mathcal{E},\mathcal{F},\Gamma)$ be as in Definition \ref{def:Product} where the factors are  $(X,d_X,\mu_X,\mathcal{E}_X,\mathcal{F}_X,\Gamma_X)$ and $(Y,d_Y,\mu_Y,\mathcal{E}_Y,\mathcal{F}_Y,\Gamma_Y)$.
    By Theorem \ref{thm:Z}, $(Z,d,\mu,\mathcal{E},\mathcal{F},\Gamma)$ is a $(Q,\beta)$-Poincar\'e--Dirichlet space satisfying the conditions in Assumption \ref{Assumption:Confdim} and its Hausdorff dimension is equal to $d_{\textup{H}} := d_{\textup{H},X} + d_{\textup{H},Y}$.
    Since $\beta = d_{\textup{H}}$ by the hypotheses, it follows from Corollary \ref{cor:confdim} that $Q$ is equal to the conformal dimension of $(Z,d)$.

    The remaining objective is to prove the non-attainment result.
    We proceed by contraposition, i.e., assume that there exists a metric $\rho$ on $Z$, quasisymmetric to $d$, so that $(Z,\rho)$ is $Q$-Ahlfors regular. The $Q$-Hausdorff measure of $(Z,\rho)$ is denoted by $\nu$. We also fix a minimal energy dominant measure $\Lambda_X$ of $(X,d_X,\mu_X,\mathcal{E}_X,\mathcal{F}_X,\Gamma_X)$ and $\Lambda_Y$ of $(Y,d_Y,\mu_Y,\mathcal{E}_Y,\mathcal{F}_Y,\Gamma_Y)$.
    By Corollary \ref{cor:Lambda-minimal} and Corollary \ref{cor:nu-minimal}, it holds that $\textup{m} := \Lambda_X \otimes \mu_Y \ll \nu$. The Radon--Nikodym derivative of $\Lambda_X \otimes \mu_Y$ with respect to $\nu$ is denoted by $h$.
    
    Since $\nu$ is a doubling measure on $(Z,d)$, it follows from the Lebesgue differentiation theorem  that $h$ has a Lebesgue point $x \in Z$ where $h(x) > 0$. After a rescaling, we may assume that $h(x) = 1$.
    By using Definition \ref{def:SSS}-(2), we get
    \begin{equation*}
        \lim_{\varepsilon \downarrow 0} \kint_{\Phi_{x,\varepsilon}(Z)} \abs{h - 1}\, \text{d}\nu = 0.
    \end{equation*}
    Then, for every Borel set $\Omega \subseteq Z$, we have
    \[
        \abs{\nu_{x,\varepsilon}(\Omega) - \textrm{m}_{x,\varepsilon}(\Omega)} \leq \frac{1}{\nu(\Phi_{x,\varepsilon}(Z))} \int_\Omega \abs{h - 1}\, \text{d}\nu \leq \kint_{\Phi_{x,\varepsilon}(Z)} \abs{h - 1}\, \text{d}\nu.
    \]
    The combination of the previous two displays shows that
    \begin{equation}\label{eq:Blow-up-TV}
        \nu_{x,\varepsilon} - \textrm{m}_{x,\varepsilon} \to 0 \text{ in total variation as $\varepsilon \downarrow 0$.}
    \end{equation}
    Now, we use the earlier lemmas to obtain suitable convergences.
    Take any sequence $\{ \varepsilon_i \}_{i = 1}^\infty \subseteq (0,1)$ with $\varepsilon \downarrow 0$.
    By Lemma \ref{lemma:Blow-up-1}, there is a Radon measure $\tilde{\Lambda}_X$ and a subsequence $\{\tilde{\varepsilon}_j\}_{j = 1}^\infty \subseteq \{\varepsilon_i\}_{i = 1}^\infty$ so that
    \begin{equation}\label{eq:Blow-up-m}
    \textrm{m}_{x,\tilde{\varepsilon}_j} \rightharpoonup \tilde{\textrm{m}} \ll \tilde{\Lambda}_X \otimes \mu_Y.    
    \end{equation}
    By Lemma \ref{lemma:Blow-up-2}, there is a metric $\tilde{\rho}$, quasisymmetric to $d$, so that $(Z,\tilde{\rho})$ is $Q$-Ahlfors regular. Additionally, by taking further subsequence if necessary, 
    \begin{equation}\label{eq:Blow-up-nu}
        \nu_{x,\tilde{\varepsilon}_j} \rightharpoonup \tilde{\nu} \text{ and } \tilde{\nu} \text{ is comparable to the $Q$-Hausdorff measure on $(Z,\tilde{\rho})$.}
    \end{equation}
 
    Finally, we obtain the chain of absolute continuities
    \begin{align*}
        \mu_X \otimes \Lambda_Y \ll \tilde{\nu} \ll \tilde{\textup{m}} \ll \tilde{\Lambda}_X \otimes \mu_Y.
    \end{align*}
    The first one follows from Proposition \ref{cor:Lambda-minimal} and Corollary \ref{cor:nu-minimal}. For the second one, we simply have $\tilde{\nu} = \tilde{\textup{m}}$ by Equations \eqref{eq:Blow-up-TV}, \eqref{eq:Blow-up-m} and \eqref{eq:Blow-up-nu}. The last one is just \eqref{eq:Blow-up-m}.
    However, since $\mu_Y \perp \Lambda_Y$ by Proposition \ref{prop:Singular}, we also have $\mu_X \otimes \Lambda_Y \perp \tilde{\Lambda}_X \otimes \mu_Y$.
    This combined with the previous display gives $\mu_X \otimes \Lambda_Y = 0$, which is obviously false. Hence, the metric  $\rho$ attaining the conformal dimension cannot exist, and this completes the proof.
\end{proof}

\subsection{Concluding remarks}\label{subsec:Remarks}

The work concludes to a discussion about examples and some further discussion about Theorem \ref{thm:non-attainment}.

\begin{remark}\label{remark:Explain}
This remark explains how to verify the hypotheses of
Theorem \ref{thm:non-attainment} for the Sierpi\'nski carpet $\mathbb{S}$.
This, in particular, proves Theorem \ref{main:Carpet}. First, $\mathbb{S}$ is bi-Lipschitz to a geodesic metric space, and the maps in the uniform scalability condition are provided by the natural similarity maps. Thus, the geometric conditions have been checked.
Next we verify the existence of the required Dirichlet space and its properties.
By the work of Murugan and Shimizu \cite{MuruganShimizu}, for every $p \in (1,\infty)$ there is a $(p,\beta_p)$-Poincar\'e--Dirichlet space $(\mathbb{S},d_\mathbb{S},\mu_\mathbb{S},\mathcal{E}_\mathbb{S},\mathcal{F}_\mathbb{S},\Gamma_\mathbb{S})$ that satisfies Assumption \ref{Assumption:Confdim} for $d_{\textrm{H}} = \log(8)/\log(3)$ and some $\beta_p \geq p$.
Note that, by the same explanation as in the remark below, we could also refer to \cite{shimizu,kigami}.
What remains is to check is that $\beta_Q = k \cdot d_{\textrm{H}}$ for some $Q > 1$, and that $\beta_Q > Q$.
First, we have $\beta_p > p$ by \cite[Theorem 2.27]{shimizu}.
In order to find $Q$, it suffices to show
\[
    \beta_p = d_{\textup{H}} \textrm{ for some } p \in (1,\infty) \text{ and } p \mapsto \beta_p \text{ is continuous in $p \in (1,\infty)$}.
\]
We then find $Q$ by a continuity argument because $\beta_{k \cdot d_{\textrm{H}}} > k \cdot d_{\textrm{H}}$ as noted above.
Now, if $\mathcal{M}_p \in (0,\infty)$ is as in \cite[Theorem 8.1]{ReplacementGraph24}, then $\beta_p = \log(8\mathcal{M}_p^{-1})/\log(3)$ and $p \mapsto \mathcal{M}_p \in (0,\infty)$ is continuous. Hence, $\beta_p$ is also continuous.
Alternatively, we could also show the continuity by using the proof of \cite[Theorem 8.1]{ReplacementGraph24} and \cite[Lemma 4.4]{BourK}.
Note that we need to use the fact that $C$ in \cite[Lemma 4.4]{BourK} can be chosen to be independent of $p$ when $p$ varies on some compact subset of $[1,\infty)$.
For the other condition in the display, note that $\beta_p = d_{\textrm{H}}$ when $p$ is equal to the conformal dimension of $\mathbb{S}$. See \cite[Remark 9.17]{MuruganShimizu}.
Finally, the continuity can also be deduced from \cite[Proposition 4.7.5]{Kigamiweighted}.
\end{remark}

\begin{remark}
    The explanation in Remark \ref{remark:Explain} extends to the Sierpi\'nski gasket and many other post-critically finite self-similar sets. The Dirichlet spaces are provided, for instance, by \cite{p-Gasket,pcf-penergy} and \cite[Theorem 4.6]{kigami}.
    See \cite[Section 9.2]{KajinoShimizu} for details about the energy measures. Continuity of $p\mapsto \beta_p$ follows, similarly to the  previous example, from \cite[Theorem 8.1]{ReplacementGraph24}. 
    This discussion also extends to many Laakso type spaces constructed in \cite{AEBS25}, such as the Laakso diamond space.
\end{remark}

\begin{remark}\label{rem:Menger}
    This remark shows that the hypotheses of Theorem \ref{thm:non-attainment} are satisfied by the 3-dimensional Menger sponge $(\mathbb{M},d_{\mathbb{M}})$.
    The argument is essentially the same as in Remark \ref{remark:Explain}. However, let us note that currently there is no result that provides suitable $p$-Dirichlet spaces for every $p\in (1,\infty)$.
    This is, nevertheless, not an issue because we need it only for specific $p$.
    It is easy to check that the conformal dimensions $Q$ of the $k$-fold Cartesian product $(Z,d)$, for $k \geq 2$, of the Menger sponge is strictly larger than the conformal dimension $Q_{\mathbb{M}}$ of the Menger sponge itself.
    In fact, since $\mathbb{M}$ contains an isometric copy of the unit interval, $Q \geq Q_{\mathbb{M}} + 1$ by \cite[Example 4.1.9]{MT}.
    Hence, the existence of the required $Q$-Dirichlet space are provided by Kigami \cite{kigami} and Shimizu \cite{shimizu}. The same argument works for every generalized Sierpi\'nski carpet in the sense of \cite{shimizu}.
\end{remark}

\begin{remark}
    We note that Theorem \ref{thm:non-attainment-prod} fails in general if the parameter $\beta$ is different for the two factors.
    This seems to be subtle detail since, for instance $\mathbb{S} \times [0,1]$ attains its conformal dimension by \cite[Example 4.1.9]{MT}, where $\mathbb{S}$ is the Sierpi\'nski carpet.
    This shows that the conclusion of Theorem \ref{thm:non-attainment-prod} is false if the parameter $\beta$ for the two factors do not agree and $\beta > Q$ for one of the factors.
    On the other hand, we do not know whether $\mathbb{S} \times \mathbb{M}$, where $\mathbb{M}$ is the Menger sponge, attains its conformal dimension or not.
\end{remark}

\bibliographystyle{acm}
\bibliography{clp}

\end{document}